\documentclass[11pt,leqno,twoside]{amsart}
\usepackage{amssymb, amsmath}
\usepackage{bm}
\numberwithin{equation}{section}
\usepackage{latexsym}
\usepackage{amsmath}
\usepackage{amssymb}
\usepackage{mathrsfs}
\usepackage{graphicx}
\usepackage{ifthen}
\usepackage{color}
\usepackage[T1]{fontenc}
\usepackage[latin1]{inputenc}

\numberwithin{equation}{section}

\usepackage{latexsym}
\usepackage{amsmath}
\usepackage{amssymb}

\newtheorem{defi}{Definition}[section]
\newtheorem{thm}[defi]{Theorem}
\newtheorem{lemma}[defi]{Lemma}
\newtheorem{corollary}[defi]{Corollary}
\newtheorem{proposition}[defi]{Proposition}
\newtheorem{remark}[defi]{Remark}

\def\0{{\bf{0}}}
\def\e{\mbox{exp}}

\def\Go{\hat{G}^{(0)}_t}
\def\Gk{\hat{G}^{(k)}_t}
\def\n{\nonumber}
\def\M{\mathcal{M}}

\newcommand{\R}{\mathbf R}

\newcommand{\N}{\mathbf N}

\def\X{\mathcal{X}}
\newcommand{\beq}{\begin{equation}}
\newcommand{\eeq}{\end{equation}}
\def \dis{\displaystyle}

\def \pt{\partial}
\def \al{\alpha}

\def\XXint#1#2#3{{\setbox0=\hbox{$#1{#2#3}{\int}$}
\vcenter{\hbox{$#2#3$}}\kern-.5\wd0}}

\begin{document}

\title[Modified Heat Kernel]
{Long-time Asymptotic profiles of the n-D heat equation and Modified heat kernels}

\author{Kana Minami}

\address{Department of Mathematical and Physical Science, Nara Women's University, 630-8506, Nara, Japan}

\email{minami@cc.nara-wu.ac.jp}

\author{Taku Yanagisawa}

\address{Department of Mathematical and Physical Science, Nara Women's University, 630-8506, Nara, Japan}

\email{taku@cc.nara-wu.ac.jp}



\maketitle

\begin{abstract}
This paper studies the long-time asymptotic behavior of solutions to the n-dimensional heat equation. We introduce modified heat 
kernels, which are scaled in size and shifted in spatial and time variables.
The scaled size, spatial and time shifts of the modified heat kernels are determined from higher-order moments of the initial data, 
by systematic use of Taylor's formula. Estimating the error terms carefully, we demonstrate that modified heat kernels provide 
fine long-time asymptotic profiles of the solution to the initial value problem of the n-dimensional heat equation.
\end{abstract}

\section{Introduction}

We study the asymptotic behavior of a solution to the $n-$dimensional ($n\geq 2$) heat equation
\begin{equation}\label{eqn:1.1}
u_t=\Delta u,\, x\in\R^n,\,t>0
\end{equation}
with the initial condition
\begin{equation}\label{eqn:1.2}
u(x,0)=f(x), \,x\in\R^n,
\end{equation}
where $u=u(x,t)=u(x_1,\dots,x_n,t)$ denotes an unknown function at $x=(x_1,\dots,x_n)\in\R^n$ and $t>0$, $f=f(x)=f(x_1,\dots,x_n)$ the initial data.

If the initial data $f$ is in $L^1(\R^n)$, the unique solution to (\ref{eqn:1.1}), (\ref{eqn:1.2}) is given explicitly by
\begin{equation}\label{eqn:1.3}
u(x,t)=\int_{\R^3} G_t(x-y) f(y)\,dy
\end{equation}
where $G_t(x)$ is the n-dimensional heat kernel defined by
\begin{equation}\label{eqn:1.4}
G_t(x)=\frac{1}{(4\pi t)^{\frac{n}{2}}}\e(-\frac{|x|^2}{4t}),\,\,\,\,|x|=\sqrt{\sum^n_{i=1} x_i^2}.
\end{equation}
It is known that the heat kernel $G_t(x)$ is a self-similar solution to (\ref{eqn:1.1}), which represents a generic long-time asymptotic profile of the solution to (\ref{eqn:1.1}), (\ref{eqn:1.2}) .
However, it might be possible to think that the solution to (\ref{eqn:1.1}), (\ref{eqn:1.2}) passing a long time shall keep its memory 
at the initial time. Therefore, long-time asymptotic behaviors of the solution to (\ref{eqn:1.1}),  
 (\ref{eqn:1.2}) shall be affected even faintly by the initial data $f$.

In this paper, we present a new ansatz which puts information on the initial data into the heat kernel. More precisely, we 
introduce $\it{modified\,heat\,kernels}$ in which the size, center, and variance of $G_t(x)$ are adjusted by higher-order 
moments of the initial data. 

We first introduce {\it the 0-th order modified heat kernel} $\Go$ given by
\begin{eqnarray}\label{eqn:1.5}
\Go(x)&=&\Go(x;x^*,t^*)\\
&=&\int_{\R^n} f(x)\,dx \,\cdot G_{t-t^*}(x-x^*) \n \\ \nonumber &=&\int_{\R^n} f(x)\,dx\,\,\frac{1}{\,(4\pi (t-t^*))^{\frac{n}{2}}\,}\,\e(-\frac{\,|x-x^*|^2\,}{4(t-t^*)})\,\, \mbox{for}\,\,t>t^*,\,\,x\in\R^n. \n \nonumber
\end{eqnarray}
Here $\int_{\R^n} f(x)\,dx$ represents the change of the size of $G_t(x)$, and $x^*=(x_1^*,\dots,x_n^*)$ and $t^*\in\R$ represent the spatial and time shifts of $G_t(x)$, respectively.

We will observe that the optimal spatial shift $x^*$ shall be determined by the first-order moments of $f$. In fact, the i-th component of the spatial shift $x^*$ is given by
\begin{equation}\label{eqn:1.6}
x^*_i \int_{\R^n} f(x)\,dx =\int_{\R^n} x_i f(x)\,dx,\,\,\,i=1,\dots,n.
\end{equation}
To determine the optimal time shift $t^*$, we find that the following additional condition for the first and second order moments of $f$ shall be required: there exists a constant $c_{\bf{0}}\in\R$ such that
\begin{equation}\label{eqn:a}\displaystyle
\frac{1}{\,2\,}\int_{\R^n} \, x_i x_j\,f(x)\,dx-\frac{1}{\,2\,}\int_{\R^n} \,f(x)\,dx\, x^*_i x^*_j=c_{\bf{0}}\,\delta_{ij}\quad \mbox{for\,all}\,\, i,j=1,\dots,n,
\end{equation}
where $x^*=(x_1^*,\dots,x_n^*)$ is the spatial shift given by (\ref{eqn:1.6}), and $\delta_{ij}$ the Kronecker's delta.

Now we provide a result stating that the 0-th order modified heat kernel gives an appropriate long-time asymptotic profile of the solution to (\ref{eqn:1.1}), (\ref{eqn:1.2}) in $\it{the\, non\!-\!degenerate\,case}$ that $\displaystyle \int_{\R^n}\,f(x)\,dx\neq 0$.
\begin{thm} Let $f\in L^1(\R^n)$ and $\int_{\R^n}\,(1+|x|)^2 |f(x)|\,dx<\infty$. Assume that
$$\int_{\R^n} f(x)\,dx\neq 0.$$
We take the spatial shift $x^*$ of $\Go(x)$ of (\ref{eqn:1.5}) as in
(\ref{eqn:1.6}). Additionally, the initial data $f$ is supposed to satisfy the condition (\ref{eqn:a}), and the time shift $t^*$ of $\Go(x)$ is taken to satisfy that
\begin{equation}\label{eqn:c}\displaystyle
c_{\bf{0}}+t^*\int_{\R^n} f(x)\,dx=0,
\end{equation}
where $c_{\bf{0}}$ is a constant in (\ref{eqn:a}).

Then the solution $u=u(x,t)$ to (\ref{eqn:1.1}), (\ref{eqn:1.2}) given by (\ref{eqn:1.3}) and $\Go(x;x^*,t^*)$ obey the estimate for $1\leq p \leq \infty$ such that
\begin{equation}\label{eqn:1.8}
\displaystyle \lim_{t\rightarrow\infty} t^{1+\frac{n}{2}(1-\frac{1}{p})}\|u(\cdot,t)-\Go(\cdot;x^*,t^*)\|_{L^p(\R^n)}=0.
\end{equation}
\end{thm}

\bigskip

\begin{remark}
(i) If the $f$ satisfies the condition that
$\int_{\R^n}(1+|x|)^3|f(x)|\,dx<+\infty$ additionally, the error estimate (\ref{eqn:1.8}) can be improved as
\begin{equation}\label{eqn:1.9}
\displaystyle \|u(\cdot,t)-\Go(\cdot;x^*,t^*)\|_{L^p(\R^n)}\leq Ct^{-\frac{3}{2}-\frac{n}{2}(1-\frac{1}{p})}
\end{equation}
for all $t>\max\{t*,\,0\}$. Here $C$ is a constant depending only on $\int_{\R^n}(1+|x|)^3|f(x)|\,dx$. For the proof of the statement (i), see the last part of Proof of Theorem 1.5 in Section 3. 

\medskip

(ii) Let $f$ satisfy the same conditions as in Theorem 1.1. It follows from (\ref{eqn:1.6}) and (\ref{eqn:a}) that
\begin{equation}\label{eqnrem:1.1}\displaystyle
\frac{\int_{\R^n} \, x_i x_j\,f(x)\,dx}{\int_{\R^n}\, f(x)\,dx}=\frac{\int_{\R^n} \,x_i \,f(x)\,dx}{\int_{\R^n}\, f(x)\,dx}\,\cdot\,
\frac{\int_{\R^n} \,x_j \,f(x)\,dx}{\int_{\R^n}\, f(x)\,dx}\quad \mbox{for}\,\,\,i\neq j,\,\,i,j=1,\dots,n.
\end{equation}
On the other hand, since the identity $\int_{\R^n}\, f(x)\,dx\, (x_i^*)^2=\int_{\R^n}\,x_i\,f(x)\,dx \,x^*_i$ holds by (\ref{eqn:1.6}), we see from (\ref{eqn:a}) that
\begin{eqnarray}\label{eqnrem:1.2}\displaystyle
c_{\bf{0}}&=&\frac{1}{\,2\,}\int_{\R^n} \, x_i^2 \,f(x)\,dx-\frac{1}{\,2\,}\int_{\R^n} \,f(x)\,dx\, (x^*_i)^2\nonumber \\
&=&
\frac{1}{\,2\,}\int_{\R^n} \, x_i^2 \,f(x)\,dx-\int_{\R^n} \,x_i\,f(x)\,dx\, x^*_i +\frac{1}{\,2\,}\int_{\R^n} \,f(x)\,dx\, (x^*_i)^2 \\
&=& \frac{1}{\,2\,}\int_{\R^n}\,(x_i-x_i^*)^2\,f(x)\,dx\quad \mbox{for}\,\,\,i=1,\dots,n. \nonumber
\end{eqnarray}
Then, using the notation that $\displaystyle <h(x)>_f=\frac{\int_{\R^n} h(x)f(x)\,dx}{\int_{\R^n} f(x)\,dx}$ for a function $h(x)$ defined on $\R^n$, we can see that the conditions (\ref{eqnrem:1.1}), (\ref{eqnrem:1.2}) can be rewritten as
\begin{eqnarray*}
<x_i \,x_j>_f&=&<x_i>_f<x_j>_f \quad \mbox{for}\,\,i\neq j, \,\,i,j=1,\dots,n,\\
\displaystyle \frac{1}{2}<(x_i-x^*_i)^2>_f&=&<c_{\bf{0}}>_f\quad \mbox{for}\,\, i=1,\dots,n.\\
\end{eqnarray*}
Hence, we might say that the conditions (\ref{eqnrem:1.1}) and (\ref{eqnrem:1.2}) imply ``the independence of each component of displacement'' and ``the equi-variance of all components of displacement'' with respect to the mean value with the weight $f$, respectively.

\medskip

(iii) We can see from (\ref{eqn:a}) with $i=j$ and (\ref{eqnrem:1.2}) that the time shift $t^*$ defined by (\ref{eqn:c}) obeys the 
equality
\begin{equation}\label{eqn:1.7}
2n t^* \int_{\R^n} f(x)\,dx =-\sum^n_{i=1} \int_{\R^n} (x_i-x^*_i)^2 f(x)\,dy.
\end{equation}

\medskip

(iv) Even when the additional condition (\ref{eqn:a}) does not hold, by taking just the spatial shift $x^*$ as in (\ref{eqn:1.6}) and the time shift as $t^*=0$, the modified heat kernel $\Go(x;x^*,0)$ gives a better long-time asymptotic profile than the original heat kernel: if
$f\in L^1(\R^n)$, $\int_{\R^n}(1+|x|)|f(x)|\,dx<\infty$ and $\int_{\R^n} f(x)\,dx\neq 0$, then the estimate
\begin{equation}\label{eqn:1.8}
\displaystyle \lim_{t\rightarrow\infty} t^{\frac{1}{2}+\frac{n}{2}(1-\frac{1}{p})}\|u(\cdot,t)-\Go(\cdot;x^*,0)\|_{L^p(\R^n)}=0
\end{equation}
holds.

\medskip

(v) There exist nontrivial functions satisfying the condition (\ref{eqn:a}). In fact, one of such examples is given in the form
\begin{align*}
f(x)=\prod_{i=1}^n \, a_i(x_i),
\end{align*}
where
$$
a_i(x)=(c_{i,0}+c_{i,1}\,x) \, e^{-x^2}, \qquad c_{i,0}\in \R\setminus\{0\} ,\,\,c_{i,1} \in \R \; , \; i=1, \dots , n.
$$
We find that $f\in L^1(\R^n)$ and $\int_{\R^n}\,(1+|x|)^2\,f(x)\,dx<\infty$ and $\int_{\R^n}\,f(x)\,dx\neq 0$.
If $c_{i,0}$, $c_{i,1}$ are taken as $\dis \frac{c_{i,1}}{c_{i,0}} =c$ for some constant $c \in \R$, $i=1,\dots,n,\,$ the condition (\ref{eqn:a}) holds and the spatial and time shifts are given explicitly by
\begin{align*}
x_i^*=\frac{c}{\,2\,}\,\,\mbox{for} \;\; i=1,\dots n, \quad t^*=-\frac{1}{\;4\;} \, \left(1-\frac{c^2}{\,2\,} \right).
\end{align*}
See Appendix D for more details.
\end{remark}

\medskip

When we integrate $u$ over the unit sphere in $\R^n$, we need not take the spatial shift
$x^*$ of $\Go(x)$ as in (\ref{eqn:1.6}) and gain one half decay order in the error estimate than that of Theorem 1.1.
\begin{thm} Let $f\in L^1(\R^n)$ and $\int_{\R^n}(1+|x|)^3|f(x)|\,dx<\infty$.
Assume that
$$
\int_{\R^n} f(x)\,dx\neq 0.
$$
Additionally, the initial data $f$ is supposed to satisfy the conditions (\ref{eqn:a}). We take the time shift $t^*$ as in 
(\ref{eqn:c}) and put the spatial shift $x^*=0$ of $\Go(x)$.

Then the solution $u=u(x,t)$ to (\ref{eqn:1.1}), (\ref{eqn:1.2}) given by (\ref{eqn:1.3}) and $\Go(x;0,t^*)$ obey the estimate for $1\leq p \leq \infty$ such that
\begin{equation}\label{eqn:1.10}
\displaystyle \lim_{t\rightarrow\infty} t^{\frac{3}{2}+\frac{n}{2}(1-\frac{1}{p})}\|[u(\cdot,t)]-\Go(\cdot;0,t^*)\,|S^{n-1}|\,\|_{L^p(\R^n)}=0.
\end{equation}
Here $[u]$ denotes the integration of $u$ over the unit sphere $S^{n-1}$ in $\R^n$, that is,
$$
[u]=\int_{S^{n-1}} u\,d\omega,
$$
and
$|S^{n-1}|$ denotes the measure of $S^{n-1}$ which is given by
$$\displaystyle
|S^{n-1}|=\frac{2\pi^{\frac{n}{2}}}{\,\Gamma(\frac{n}{2})\,}.
$$
\end{thm}
\begin{remark}
(i) We notice that if we take $x^*=0$ in $\Go$, then $\Go$ depends only on $r$ with $r=|x|$, so that
$$
[\Go(x;0,t^*)]=\Go(x;0,t^*)|S^{n-1}|.
$$

\medskip

(ii) If the initial data $f\in L^1(\R^n)$ is radially symmetric, the solution $u$ to (\ref{eqn:1.1}), (\ref{eqn:1.2}) shall be radially symmetric. Furthermore, if the radially symmetric initial data $f\in L^1(\R^n)$ satisfies that
$\int_{\R^n} (1+|x|)^3 |f(x)| dx<\infty$ , it holds that
$\int_{\R^n} x_i f(x) dx=0, \int_{\R^n} x_i x_j f(x) dx=c_f\delta_{ij}$ for a specific constant 
$c_f\in\R, \, \mbox{for\,all}\,\, i,j=1,\dots,n$.
Hence, the condition (\ref{eqn:a}) always holds for $\displaystyle c_0=\frac{c_f}{2}$. 
Therefore, for each radially symmetric initial data $f\in L^1(\R^n)$ such that 
$\int_{\R^n} (1+|x|)^3 |f(x)| dx<\infty, \,\int_{\R^n} f(x) dx\neq 0$,
the error estimate (\ref{eqn:1.10}) for $[u]$ turns out to be
\begin{equation*}
\displaystyle \lim_{t\rightarrow\infty} t^{\frac{3}{2}+\frac{n}{2}(1-\frac{1}{p})}\|u(\cdot,t)
-\Go(\cdot;0,t^*)\,\|_{L^p(\R^n)}=0,
\end{equation*}
where $\displaystyle t^*=-\frac{c_f}{2\,\int_{\R^n} f(x) dx\,}$.
\end{remark}

\medskip

On the other hand, in $\it{the \,degenerate\, case}$ that $\int_{\R^n} f(x)\,dx= 0$, it is clear that both the spatial and time shifts $x^*$ and $t^*$ of $\Go$ cannot be determined by the equalities (\ref{eqn:1.6}) and (\ref{eqn:c}). To overcome this difficulty, we now introduce {\it the k-th order modified heat kernel} with $k\in\N$.
In what follows, the initial data $f$ is assumed to satisfy that $f\in L^1(\R^n)$ and $\int_{\R^n}(1+|x|)^{k+2}|f(x)|\,dx<+\infty$,
unless otherwise noticed. For $\ell \in\N_0$, set 
$$
\Lambda_\ell(f)=\{\,\alpha\in(\N_0)^n\, \,|\, \,\M_\alpha(f)\neq 0 \, \mbox{with} \, |\alpha|=\ell \,\},
$$
where
$\M_\alpha(f)$ is the $\alpha$-moment of the initial data $f$ defined by
$$
\M_\alpha(f)=\int_{\R^n} x^\alpha f(x)\,dx.
$$
For $k\in \N$, {\it the k-th order modified heat kernel} $\Gk$ is defined by
\begin{eqnarray}\label{eqn:1.12}
\Gk(x)&=&\Gk(x; x^*,t^*) \nonumber \\
\displaystyle &=&\sum_{0\leq |\alpha|\leq k-1} \frac{\,(-1)^{|\alpha|\,}}{\alpha!} \M_\alpha(f)\partial_x^\alpha G_t(x)+
\sum_{\alpha \in \Lambda_k(f)}\frac{\,(-1)^{k\,}}{\alpha!} \M_\alpha(f)\partial_x^\alpha G_{t-t^{*,\alpha}}(x-x^{*,\alpha})\\
&\phantom{}&\qquad \qquad \qquad \qquad \qquad \qquad \qquad \qquad
\mbox{for}\,\,t>\max\,\{\,t^{*,\alpha}, 0\,\,|\, \alpha\in \Lambda_k(f)\}, \,\,x\in\R^n, \nonumber
\end{eqnarray}
where $x^*=\{\,x^{*,\alpha}\,|\,\alpha\in\Lambda_k(f)\,\}$ and $t^*=\{\,t^{*,\alpha}\,|\, \alpha\in\Lambda_k(f)\,\}$ denote the
total spatial and time shifts of $\Gk(x)$, respectively. More concretely,
$x^{*,\alpha}=(x^{*,\alpha}_1,\dots,x^{*,\alpha}_n)$ and $t^{*,\alpha}$ denote the spatial and time shifts of
$\partial_x^\alpha G_t(x)$ for $\alpha\in \Lambda_k(f)$. 
For all $\alpha\in\Lambda_k(f)$, we take spatial shifts $x^{*,\alpha}$ in (\ref{eqn:1.12}) as
\begin{equation}\label{eqn:1.13}\displaystyle
x^{*,\alpha}_i=\frac{\M_{\alpha+e_i}(f)}{\,\,(k+1)\M_\alpha(f)\,\,}, \,\,\,i=1,\dots,n,
\end{equation}
where $e_i$ is the i-th unit vector in $\R^n$.
Meanwhile, to define the time shifts $t^{*,\alpha}$ for all $\alpha\in\Lambda_k(f)$, we need to assume that the k-th, (k+1)-th and (k+2)-th order moments of the initial data $f$ satisfy the following

\medskip

$\bf{Condition \, A.}$
For all $\alpha\in\Lambda_k(f)$, there exist constants $c_\alpha\in\R$ such that

\begin{equation}\label{eqn:1.a}
\displaystyle
\frac{1}{\,(k+1)(k+2)\,} \M_{\alpha+e_i+e_j}(f)-\frac{1}{\,2!\,}\M_\alpha(f)(x^{*,\alpha})^{e_i+e_j}=c_\alpha \delta_{ij}\,\,\,
\mbox{for\,all}\,\,i,j=1,\dots,n.
\end{equation}
Here, $c_\alpha$ are constants dependent only on $\alpha$ and $f$, while $x^{*,\alpha}$ are the spatial shifts given 
by (\ref{eqn:1.13}).

Furthermore, for all $\alpha\notin\Lambda_k(f)$ with $|\alpha|=k$, the moments of the initial data $f$ satisfy that
\begin{equation}\label{eqn:1.c}
\displaystyle
\M_{\alpha+e_i}(f)=\M_{\alpha+\alpha'}(f)=0\,\,\,
\mbox{for\,all}\,\,i=1,\dots,n,\,\,\mbox{and\, all }\,\alpha'\in (\N_0)^n\,\mbox{with}\,|\alpha'|=2.
\end{equation}

\medskip

Our main theorem of this paper reads
\begin{thm} Let $k\geq 1$ be an integer. Let $f\in L^1(\R^n)$ and $\int_{\R^n}(1+|x|)^{k+2}|f(x)|\,dx<\infty$.
Assume that
$$
\Lambda_k(f)\neq\phi.
$$
We take the spatial shifts $x^{*,\alpha}$, $\alpha \in \Lambda_k(f)$, of $\Gk(x)$ in (\ref{eqn:1.12}) as in (\ref{eqn:1.13}).
Additionally, the initial data $f$ is supposed to fulfill (\ref{eqn:1.a}), (\ref{eqn:1.c}) in Condition A, and the time shifts
$t^{*,\alpha}$, $\alpha\in \Lambda_k(f)$, of $\Gk(x)$ are taken to satisfy that
\begin{equation}\label{eqn:1.b}\displaystyle
c_\alpha+t^{*,\alpha}\M_\alpha(f)=0,
\end{equation}
where $c_\alpha$ are the same constants as in (\ref{eqn:1.a}).

Then the solution $u=u(x,t)$ to (\ref{eqn:1.1}), (\ref{eqn:1.2}) given by (\ref{eqn:1.3}) and $\Gk(x;x^*,t^*)$ obey the estimate for 
$1\leq p \leq \infty$ such that
\begin{equation}\label{eqn:1.16}
\displaystyle \lim_{t\rightarrow\infty} t^{\frac{k+2}{2}+\frac{n}{2}(1-\frac{1}{p})}\|u(\cdot,t)-\Gk(\cdot;x^*,t^*)\|_{L^p(\R^n)}=0.
\end{equation}
\end{thm}

\medskip

\begin{remark}
(i) As in Remark 1.2 (iv), we have the following statement. If $f\in L^1(\R^n)$, $\int_{\R^n} (1+|x|)^{k+1} |f(x)|dx<\infty$ and $\Lambda_k(f)\neq\phi$, then the $k$-th order modified heat kernel $\Gk(x)$ with the spatial shifts $x^{*,\alpha}$ given by (\ref{eqn:1.13}) and the time shifts $t^{*,\alpha}=0$, $\alpha\in\Lambda_k(f)$, obeys the error estimate
\begin{equation}\label{eqn:1.17}
\displaystyle \lim_{t\rightarrow\infty} t^{\frac{k+1}{2}+\frac{n}{2}(1-\frac{1}{p})}\|u(\cdot,t)-\Gk(\cdot;x^*,0)\|_{L^p(\R^n)}=0.
\end{equation}

\medskip

(ii) We can see from (\ref{eqn:1.13}) and (\ref{eqn:1.a}) with $i=j$ that the time shift $t^{*,\alpha}$ for an 
$\alpha\in\Lambda_k(f)$ defined by (\ref{eqn:1.b}) satisfies the equality such that
\begin{equation}\label{eqn:1.d}
n(k+1)(k+2)\M_\alpha(f) t^{*,\alpha} =-\sum^n_{i=1} \int_{\R^n} (x_i-s\, x^{*,\alpha}_i)^2 \,x^\alpha f(x)\,dy,
\end{equation}
where
$$\displaystyle
s=\frac{\,2(k+1)+\sqrt{2k(k+1)}\,}{2}.
$$
The equality (\ref{eqn:1.7}) in Remark 1.2 is obtained by putting $k=0$ in (\ref{eqn:1.d}).
See Appendix B for the proof of (\ref{eqn:1.d}).

\medskip

(iii) There exists a nontrivial function $f$ that satisfies Condition A, for example, given in the form such that
\begin{align}\label{eq:2.15}
f(x)=\prod_{i=1}^n \, a_i(x_i),
\end{align}
where
$$
a_i(x)=\sum_{j=0}^{k+1}(c_{i,j}\,x^{j}) \, e^{-x^2}, \quad c_{i,j} \in \R,\,i=1,\dots,n,\,j=0,1,\dots,k+1.
$$
In particular, when $k=1$, taking $\dis (c_{1,0}\,,c_{1,1}\,,c_{1,2})=(\mathcal{C}_1 , \pm\sqrt{2}\,\mathcal{C}_1,-2\,\mathcal{C}_1)$
and 
$\dis(c_{i,0}\,,c_{i,1}\,,c_{i,2})=(\mathcal{C}_i ,0,-\frac{2}{3}\,\mathcal{C}_i)$, $i=2,\dots,n$, for arbitrary constants
$\mathcal{C}_i\in\R\setminus\{0\}$, $i=1,2,\dots,n$, we can show that the function
$f$ in (\ref{eq:2.15}) satisfies Condition A, and the spatial and time shifts $x^{*,\alpha}$ and $t^{*,\alpha}$, 
$\alpha\in\Lambda_1(f)$, are given by
$$ \Lambda_{1}(f)=\{e_1\}, \,\, x_1^{*,e_1}=\mp\sqrt{2},\,\,x_i^{*,e_1}=0\,\,
\mbox{for}\,\,i=2,\dots,n,\,\,t^{*,e_1}=0.
$$
See Appendix E for more details.

\end{remark}

\bigskip

Let $C_c(\R^n)$ be the space of all compactly supported continuous functions on $\R^n$.

Then, Stone-Weierstrass theorem readily leads to
\begin{proposition}%Proposition 1.7%
Let $f\in C_c(\R^n)$. If $\mathcal{M}_\alpha(f)=0$ for all $\alpha\in (\N_0)^n$, then $f\equiv 0$.
\end{proposition}
The proof of Proposition 1.7 is given in Appendix C. According to Proposition 1.7, we have 
\begin{proposition}%Proposition 1.8%
Suppose that $f\in C_c(\R^n)$ is nontrivial. Then the following two cases happen: 
(I) $k_0=\max \{ k\in \N_0\,|\, \Lambda_k(f)\neq \phi\,\}$ exists. 
(II) $\sup \{ k\in \N_0\, | \, \Lambda_k(f)\neq \phi\,\}= \infty$.
\end{proposition}
Combining Theorem 1.5 and Proposition 1.8, we obtain
\begin{corollary}%Corollary 1.8%
Suppose that the initial data $f\in C_c(\R^n)$ is nontrivial.

In case (I) of Proposition 1.8, the following statement holds: 
Assume that $f$ satisfies Condition A with $k=k_0$. Then the solution $u=u(x,t)$ to (\ref{eqn:1.1}), (\ref{eqn:1.2}) given by 
(\ref{eqn:1.3}) and the $k_0$-th order modified heat kernel $\hat{G}_t^{(k_0)}(x)$ of
(\ref {eqn:1.12}) with the spatial and time shifts
$x^{*,\alpha}$, $t^{*,\alpha}$, $\alpha\in\Lambda_{k_0}(f)$ defined
by (\ref{eqn:1.13}), (\ref{eqn:1.b}) obey the 
estimate for $1\leq p \leq \infty$ such that
\begin{equation}\label{eqn:1.24}
\displaystyle \lim_{t\rightarrow\infty} t^{\frac{k_0+2}{2}+\frac{n}{2}(1-\frac{1}{p})}
\|u(\cdot,t)-\hat{G}_t^{(k_0)}(\cdot;x^*,t^*)\|_{L^p(\R^n)}=0.
\end{equation}

In case (II) of Proposition 1.8, there exists a sequence $\{k_i\}\subset \N$ with $\lim_{i\rightarrow\infty} k_i=\infty$ such that 
$\Lambda_{k_i}(f)\neq \phi,\, i=1,\dots$. Then each $k_i$-th order modified heat kernel $\hat{G}_t^{(k_i)}(x)$ of (\ref{eqn:1.12}) with 
the spatial $x^{*,\alpha}$, $\alpha \in \Lambda_{k_i}(f)$, taken as in (\ref{eqn:1.13}) and trivial total time shifts $t^*=0$, and the
solution $u=u(x,t)$ to (\ref{eqn:1.1}), (\ref{eqn:1.2}) given by (\ref{eqn:1.3}) obey the estimate for $1\leq p \leq \infty$ such that
\begin{equation}\label{eqn:1.25}
\displaystyle \lim_{t\rightarrow\infty} t^{\frac{k_i+1}{2}+\frac{n}{2}(1-\frac{1}{p})}\|u(\cdot,t)-\hat{G}_t^{(k_i)}(\cdot; x^*,0))\|
_{L^p(\R^n)}=0\,\,\,\mbox{for}\,\,i=1,2,\dots.
\end{equation}
\end{corollary}

Finally, we present a generalization of Theorem 1.3 to the degenerate case.
\begin{thm}%Theorem 1.9%
Let $k\geq 1$ be an integer, and let $f\in L^1(\R^n)$, $\Lambda_k(f)\neq \phi$.

\mbox{(I)} When $k$ is even, we assume that
$$
\int_{\R^n} (1+|x|)^{k+3} |f(x)|\,dx <\infty
$$
and $f$ satisfies (\ref{eqn:1.a}) in Condition A and $\M_{\alpha+\alpha'}(f)=0$ for all $\alpha\notin \Lambda_k(f)$
with $|\alpha|=k$ and all $\alpha'\in(\N_0)^n$ with $|\alpha'|=2$. Then the k-th order modified heat kernel $\Gk(x)$ of
(\ref{eqn:1.12}) with the time shifts $t^{*,\alpha}$,
$\alpha\in\Lambda_k(f)$, defined by (\ref{eqn:1.b}) and trivial total spatial shift $x^*=0$, and the solution $u=u(x,t)$ to 
(\ref{eqn:1.1}), (\ref{eqn:1.2}) given by (\ref{eqn:1.3}) obey the estimate for $1\leq p \leq \infty$ such that
\begin{equation}\label{eqn:1.26}
\displaystyle \lim_{t\rightarrow\infty} t^{\frac{k+3}{2}+\frac{n}{2}(1-\frac{1}{p})}
\|[u(\cdot,t)]-\Gk(\cdot; 0,t^*)|S^{n-1}|\,\|_{L^p(\R^n)}=0.
\end{equation}

\medskip

(II) When $k$ is odd, we assume that
$$
\int_{\R^n} (1+|x|)^{k+2} |f(x)|\,dx <+\infty
$$
and $f$ satisfies that $\M_{\alpha+e_i}(f)=0$ for all $\alpha\notin \Lambda_k(f)$ with $|\alpha|=k$ and all $i=1,\dots,n$.
Then the k-th order modified heat kernel $\Gk(x)$ of (\ref{eqn:1.12}) with the spatial shifts $x^{*,\alpha}$,
$\alpha\in\Lambda_k(f)$, defined by (\ref{eqn:1.13}) and trivial total time shift $t^*=0$, and the solution $u=u(x,t)$ to
 (\ref{eqn:1.1}), (\ref{eqn:1.2}) given by (\ref{eqn:1.3}) obey the estimate for $1\leq p \leq \infty$ such that
\begin{equation}\label{eqn:1.27}
\displaystyle \lim_{t\rightarrow\infty} t^{\frac{k+2}{2}+\frac{n}{2}(1-\frac{1}{p})}\|[u(\cdot,t)]-[\Gk(\cdot; x^*,0)]\|_{L^p(\R^n)}=0.
\end{equation}
\end{thm}

\medskip

Long-time asymptotic behaviors of the solutions to the evolution equations of diffusion type have been extensively studied in 
many papers by various methods. Among others, Kleinstein and Ting \cite{KlTi}
and Witelski and Bernoff \cite{WiBe} introduced the modified heat kernel in the non-degenerate case
as a long-time asymptotic profile to the heat equation. They determined the three parameters of the modified heat kernel formally through the eigenfunction expansion around the self-similar solution, although no precise error terms were given there.
In Yanagisawa \cite{Ya}, for the 1-D heat equation, modified heat kernels were introduced in the degenerate 
case and exact error estimates were given.

Ishige, Kawakami and Kobayashi \cite{IsKaKo} studied a nonlinear integral equation, including the heat equation as a special case.
They introduced a crucial time-dependent operator which maps the initial data to the function of which all moments vanish for 
each $t>0$ in order to present a long-time asymptotic profile behaving like a multiple of the integral kernel. 
Finally, we mention that Fujigaki and Miyakawa \cite{FuMi} gave pioneering results on long-time asymptotic profiles of the solution 
to the initial value problem for the incompressible Navier-Stokes equation. We note that incompressibility of the initial velocity 
generally implies the degeneracy of the 0-th order moment of the each component.  

This paper is organized as follows: In Section 2, we shall derive an explicit form of the error term of the heat solution given by 
(\ref{eqn:1.3}) for the k-th order modified heat kernel $\Gk(x)$ in (\ref{eqn:1.12}),  using Taylor's formula systematically. 
Through this process, we will explain the reason why Condition A shall be required for initial higher-order moments and the total 
spatial and time shifts $x^*$ and $t^*$ of the k-th order modified heat kernel $\Gk(x)$ shall be taken as in (\ref{eqn:1.13}) and 
(\ref{eqn:1.b}), respectively.  
Section 3 provides the proof of Theorem 1.5, in which the proof of Theorem 1.1 is included as a special case. The error term 
given in the preceding section is estimated in $L^p(\R^n)$ carefully. We next prove 
Theorem 1.10, which includes Theorem 1.3 as the case when $k=0$, based on two elementary observations:
the first one relates to the integral of polynomials on $S^{n-1}$ (see Lemma 3.2) and the second one concerns an explicit form 
of the spatial derivatives of the heat kernel $G_t(x)$ (see Lemma 3.3).

Throughout this paper, we use notations such that for a multi-index $\alpha =(\alpha_1,\dots,\alpha_n)$,
$\alpha_i\in \N_0(=\N\cup\{0\}), \, i=1,\dots,n$, 
$\displaystyle \partial_x^\alpha=\frac{\partial^{|\alpha|}}{\partial {x_1}^{\alpha_1} \dots \partial {x_n}^{\alpha_n}}$
with $|\alpha|=\sum^n_{i=1}\alpha_i$, and $x^\alpha={x_1}^{\alpha_1} \dots {x_n}^{\alpha_n}$, $\alpha !=\alpha_1 ! \dots \alpha_n !$.
We set for abbreviation $\displaystyle \partial_{x_i}=\frac{\partial}{\partial x_i}, i=1,\dots,n, \,\partial_t=\frac{\partial}{\partial t}$.

\section{Derivation of  an explicit form of the error term}
We assume here that the initial data $f(x)$ decays sufficiently fast at infinity. 
Applying Taylor's formula to the kernel $G_t(x-y)$ of (\ref{eqn:1.3}) around $-y=0$, we see that the heat solution $u(x,t)$ given 
by (\ref{eqn:1.3}) is expanded as, for any $k\in\N_0$, 
\begin{multline}\label{eqn:2.1}
u(x,t) = \int_{\R^n} G_t(x-y) f(y)\,dy  \\
    = \displaystyle \sum^{k+2}_{i=0} \frac{1}{\,i!\,}\int_{\R^n}\, L_n(D)^i\,G_t(x)\,f(y)\,dy+\int_{\R^n}\,R_{k+2} (G_t)(x,y)\,f(y) \,dy \hspace{20pt} \\
 = \displaystyle \sum_{|\alpha|\leq k-1}\frac{\,(-1)^{|\alpha|}\,}{\alpha!}\M_\alpha(f)\partial^\alpha_x G_t(x) +
 \sum_{\alpha \in \Lambda_k(f) }\frac{\,(-1)^{|\alpha|}\,}{\alpha!}\M_\alpha(f)\partial^\alpha_x G_t(x) \hspace{8pt} \\
  \displaystyle+\frac{1}{\,(k+1)!\,}\int_{\R^n}\, L_n(D)^{k+1}\,G_t(x)\,f(y)\,dy +\frac{1}{\,(k+2)!\,}\int_{\R^n}\, L_n(D)^{k+2}\,G_t(x)\,f(y)\,dy\\
 \displaystyle +\int_{\R^n}\,R_{k+2} (G_t)(x,y) \,f(y) \,dy, 
\end{multline}
where 
$$
L_n(D)=\sum^n_{i=1}(-y_i)\frac{\partial}{\partial x_i}
$$ 
and the remainder term $R_{\ell}(G_t)(x,y)$, $\ell \in \N$, is defined by 
\begin{eqnarray*}\label{eqn:2remainder}
R_{\ell}(G_t)(x,y)&=&\frac{1}{\,(\ell-1)!\,}\int^1_0 (1-\theta)^{\ell-1} \left(\frac{d}{d\theta}\right)^\ell G_t(x-y\theta)\,d\theta \n\\
&=& (-1)^\ell \, \ell \int^1_0  (1-\theta)^{\ell-1} \sum_{|\alpha|=\ell}\frac{1}{\alpha !} \Bigl( (\partial^\alpha_x G_t)(x-y\theta)-(\partial^\alpha_x G_t)(x)  \Bigr) y^\alpha\,d\theta \n \\
\end{eqnarray*}
When $k=0$, the summation over $|\alpha|\leq k-1$ of (\ref{eqn:2.1}) shall be removed.  
We further observe that the next two terms of  (\ref{eqn:2.1}) are expressed as follows:
\begin{eqnarray}\label{eqn:2.2}
 \frac{1}{\,(k+1)!\,}\int_{\R^n}\, L_n(D)^{k+1}\,G_t(x)\,f(y)\,dy&=& \displaystyle 
 \frac{1}{\,(k+1)!\,}\int_{\R^n}\, L_n(D)^{k}\left(\sum^n_{i=1}(-y_i)\partial_{x_i}\right)\,G_t(x)\,f(y)\,dy \nonumber \\
 &=&\displaystyle \frac{-1}{\,(k+1)\,}\sum_{|\alpha|=k}\sum_{i=1}^n \frac{(-1)^k}{\alpha!}
\partial_x^{\alpha+e_i}G_t(x)\M_{\alpha+e_i}(f)
 \end{eqnarray}
 and
\begin{eqnarray}\label{eqn:2.3}
 \frac{1}{\,(k+2)!\,}\int_{\R^n}\, L_n(D)^{k+2}\,G_t(x)\,f(y)\,dy&=& \displaystyle 
 \frac{1}{\,(k+2)!\,}\int_{\R^n}\, L_n(D)^{k}\left(\sum^n_{i=1}(-y_i)\partial_{x_i}\right)^2\,G_t(x)\,f(y)\,dy \nonumber \\
 &=&\displaystyle \frac{1}{\,(k+2)!\,}\int_{\R^n}\, L_n(D)^{k}\left(\sum_{|\alpha'|=2}\frac{\,2!y^{\alpha'}\,}{\alpha' !}
 \,\partial^{\alpha'}_x G_t(x)\right)\,f(y)\,dy \\
 &=& \displaystyle \frac{1}{(k+2)(k+1)}\sum_{|\alpha|=k}\sum_{|\alpha'|=2}\frac{(-1)^k}{\alpha!}\frac{\,2!\,}{\alpha' !}
 \, \M_{\alpha+\alpha'}(f)\partial_x^{\alpha+\alpha'}G_t(x).\nonumber 
 \end{eqnarray}

We next show that the k-th order modified heat kernel $\Gk(x)$ of (\ref{eqn:1.12}) turns out to be 
\begin{multline}\label{eqn:2.4}
\Gk(x)=\Gk(x; x^*,t^*)  \\ 
%&=&\displaystyle \sum_{ |\alpha|\leq k-1} \frac{\,(-1)^{|\alpha|\,}}{\alpha!} \M_\alpha(f)\partial_x^\alpha G_t(x)+
%\sum_{|\alpha|=k}\frac{\,(-1)^{|\alpha|\,}}{\alpha!} \M_\alpha(f)\partial_x^\alpha G_{t-t^{*,\alpha}}(x-x^{*,\alpha}) \nonumber \\
\hspace{12pt}=\displaystyle \sum_{0\leq|\alpha|\leq k-1}\frac{\,(-1)^{|\alpha|\,}}{\alpha!} \M_\alpha(f)\partial_x^\alpha G_t(x)
+\sum_{\alpha \in \Lambda_k(f)}\frac{\,(-1)^{|\alpha|\,}}{\alpha!} \M_\alpha(f)\partial_x^\alpha G_t(x)\\
-\displaystyle \sum_{\alpha \in \Lambda_k(f)}\sum^n_{i=1}\frac{\,(-1)^k\,}{\alpha!} \M_\alpha(f)x^{*,\alpha}_i\partial_x^{\alpha+e_i} G_t(x) 
+\sum_{\alpha \in \Lambda_k(f)}\sum_{|\alpha'|=2}\frac{(-1)^k}{\alpha!}\frac{\,1\,}{\alpha' !}\, 
\M_{\alpha}(f) (x^{*,\alpha})^{\alpha'} \partial_x^{\alpha+\alpha'}G_t(x)  \\
-\sum_{\alpha \in \Lambda_k(f)}\sum_{i=1}^n \frac{(-1)^k}{\alpha!}\M_\alpha(f) t^{*,\alpha} \,\partial^2_{x_i}\partial^\alpha_x G_t(x) +
\hat{R}_k(x,t).
\end{multline}
In fact, applying Taylor's formula to the term $\partial_x^\alpha G_{t-t^{*,\alpha}}(x-x^{*,\alpha})$ of $\Gk(x)$ around 
$-t^{*,\alpha}=0$ and $x^{*,\alpha}=0$,  using the fact that $G_t(x)$ for $t>0$ satisfies the heat equation (\ref{eqn:1.1}), 
we have 
\begin{multline}\label{eqn:2.5}
\partial^\alpha_x G_{t-t^{*,\alpha}}(x-x^{*,\alpha})=\displaystyle \partial^\alpha_x G_t(x)-\sum^n_{i=1} \partial^{\alpha+e_i}_x G_t(x) x^{*,\alpha}_i+
\sum_{|\alpha'|=2}\frac{1}{\,\alpha'!\,}\partial^{\alpha+\alpha'}_x G_t(x)(x^{*,\alpha})^{\alpha'}\\
-\frac{\partial}{\partial t}\partial_x^{\alpha} G_t(x)t^{*,\alpha}+\tilde{R}_\alpha(x,t)\\
\hspace{98pt}=\displaystyle \partial^\alpha_x G_t(x)-\sum^n_{i=1} \partial^{\alpha+e_i}_x G_t(x) x^{*,\alpha}_i+
\sum_{|\alpha'|=2}\frac{1}{\,\alpha'!\,}\partial^{\alpha+\alpha'}_x G_t(x)(x^{*,\alpha})^{\alpha'}\\
-\sum_{i=1}^n\,\partial_{x_i}^2\partial_x^{\alpha} G_t(x)t^{*,\alpha}+\tilde{R}_\alpha(x,t).
\end{multline}    
Here the remainder terms $\hat{R}_k(x,t)$ of (\ref{eqn:2.4}) and $\tilde{R}_\alpha(x,t)$ of (\ref{eqn:2.5}) are given by 
\begin{equation}\label{eqn:2.5.1}
\hat{R}_k(x,t)=\displaystyle \sum_{\alpha \in \Lambda_k(f)}\frac{(-1)^k}{\,\alpha!\,}\M_{\alpha}(f)\tilde{R}_\alpha(x,t)
\end{equation}
and 
\begin{multline*}
\tilde{R}_\alpha(x,t)=T_1(\partial^\alpha_xG_t)(x-x^{*,\alpha},t,t^{*,\alpha})+R_2(\partial^\alpha_x G_t)(x,x^{*,\alpha})+
R_1(\partial_t\partial^\alpha_x G_t)(x,x^{*,\alpha})(-t^{*,\alpha}),
\end{multline*}
where 
$$\displaystyle 
T_\ell(\partial^\alpha_xG_t)(x,t,t')=\int^1_0 \,(1-\theta)\Bigl\{ \left(\partial^\ell_t(\partial^\alpha_x G_t)\right)_{t-\theta t'}(x)
-\partial^\ell_t (\partial^\alpha_x G_t(x))\Bigr\}(-t')\,d\theta \,\,\,\,\mbox{for}\,\,\ell\in\N
$$
with 
$$\displaystyle
\left(\partial^\ell_t(\partial^\alpha_x G_t)\right)_\tau(x)=\partial^\ell_t(\partial^\alpha_x G_t(x))|_{t=\tau}.
$$

It follows from (\ref{eqn:2.1}), (\ref{eqn:2.2}), (\ref{eqn:2.3}) and (\ref{eqn:2.4}) that the error term of 
the heat solution $u(x,t)$ of (\ref{eqn:1.3}) for the k-th order modified heat kernel $\Gk(x)$ is explicitly given by 
\begin{equation}\label{eqn:2.6}
u(x,t)-\Gk(x)
=I_1+I_2+I_3+\int_{\R^n}\,R_{k+2} (G_t)(x,y) \,f(y) \,dy-\hat{R}_k(x,t),
\end{equation}
where
\begin{multline}\label{eqn:2.7}
I_1=\displaystyle \sum_{\alpha \in \Lambda_k(f)}\, \frac{(-1)^k}{\alpha!}\,\sum_{i=1}^n
\left\{-\,\frac{1}{\,(k+1)\,}\M_{\alpha+e_i}(f)+\M_\alpha(f)x_i^{*,\alpha}\right\}\partial_x^{\alpha+e_i}G_t(x),\\
\end{multline}
\begin{multline}\label{eqn:2.8}
I_2=\displaystyle \sum_{\alpha \in \Lambda_k(f)}\,\frac{(-1)^k}{\alpha!}\,\Biggl\{\,\sum_{|\alpha'|=2}\frac{2!}{\,\alpha'!\,}
\left( \frac{1}{\,(k+1)(k+2)\,}\M_{\alpha+\alpha'}(f)-\frac{1}{\,2!\,}\M_\alpha(f)(x^{*,\alpha})^{\alpha'}\right)
\partial^{\alpha+\alpha'}_x G_t(x) \\
\displaystyle +\M_\alpha(f) t^{*,\alpha}\left(\sum^n_{i=1}\partial_{x_i}^2\partial^\alpha_xG_t(x)\right)\Biggr\},
\end{multline}
and
\begin{multline}\label{eqn:2.9}
I_3=\displaystyle \sum_{|\alpha|=k,\,\,\alpha\notin \Lambda_k(f)}
\,\frac{(-1)^k}{\alpha!}\,
\Biggl\{ -\,\frac{1}{\,(k+1)\,}\sum_{i=1}^n \M_{\alpha+e_i}(f)\partial_x^{\alpha+e_i}G_t(x) \\
+\frac{1}{\,(k+1)(k+2)}\sum_{|\alpha'|=2}\frac{2!}{\,\alpha'!\,}
\M_{\alpha+\alpha'}(f) \partial^{\alpha+\alpha'}_x G_t(x)\Biggr\}.
\end{multline}
We further divide $I_2$ into two parts $I_{2,1}$ and $I_{2,2}$ such that
$$\displaystyle 
I_{2,1}=\sum_{\alpha\in\Lambda(k)}\,\frac{(-1)^k}{\alpha!}\,\sum^n_{i,j=1,\,i\neq j}
2\left(\frac{1}{\,(k+1)(k+2)\,}\M_{\alpha+e_i+e_j}(f)-\frac{1}{\,2!\,}\M_\alpha(f)(x^{*,\alpha})^{e_i+e_j} \right)\partial_{x_i}\partial_{x_j}\partial^\alpha_x G_t(x),
$$
$$\displaystyle
I_{2,2}=\sum_{\alpha\in\Lambda_k(f)}\,\frac{(-1)^k}{\alpha!}\,\sum_{i=1}^n \left\{
\left(\frac{1}{\,(k+1)(k+2)\,}\M_{\alpha+2e_i}(f)-\frac{1}{\,2!\,}\M_\alpha(f)(x^{*,\alpha})^{2e_i} \right)+\M_\alpha(f) t^{*,\alpha}\right\}\partial^2_{x_i}\partial^\alpha_x G_t(x).
$$
It follows from (\ref{eqn:2.7}) that $I_1$ vanishes if the spatial shifts $x^{*,\alpha}$ satisfy the equalities that 
\begin{equation}\label{eqn:2.10}
-\,\frac{1}{\,(k+1)\,}\M_{\alpha+e_i}(f)+\M_\alpha(f)\,x_i^{*,\alpha}=0,\quad i=1,\dots,n, \, \,\mbox{for\,all }\,\,\alpha \in\Lambda_k(f).
\end{equation}
While, in view of $I_{2,1}$ and $I_{2,2}$, it seems natural to impose that the k-the, (k+1)-th and (k+2)-th order of 
the moments of the initial data $f$ satisfy the condition that for all $\alpha\in\Lambda_k(f)$, there exist constants 
$c_\alpha\in\R$ depending only on $\alpha$ such that 
\begin{equation}\label{eqn:2.11}
\displaystyle 
\frac{1}{\,(k+1)(k+2)\,} \M_{\alpha+e_i+e_j}(f)-\frac{1}{\,2!\,}\M_\alpha(f)(x^{*,\alpha})^{e_i+e_j}=c_\alpha \delta_{ij}\,\,\,
\mbox{for\,all}\,\,i,j=1,\dots,n.
\end{equation}  
Here $x^{*,\alpha}$ are the spatial shifts given by (\ref{eqn:2.10}). 
In fact, it is easy to see that $I_{2,1}$ vanishes under the condition (\ref{eqn:2.11}) and $I_{2,2}$ also vanishes 
if the time shifts $t^{*,\alpha}\in\Lambda_k(f)$ are taken as 
\begin{equation}\label{eqn:2.12}
\displaystyle 
c_\alpha+\M_\alpha(f)\,t^{*,\alpha}=0 \,\,\,\mbox{for}\,\,\,\alpha\in\Lambda_k(f),
\end{equation}
where $c_\alpha$ are the constants appearing in (\ref{eqn:2.11}).
Finally, we find that $I_3$ vanishes if  
\begin{equation}\label{eqn:2.13}
\displaystyle 
\M_{\alpha+e_i}(f)=\M_{\alpha+\alpha'}(f)=0 
\end{equation}
for all $\alpha\notin\Lambda_k(f)$ with $|\alpha|=k$, and for all $\alpha'$ with $|\alpha'|=2$, $i=1,\dots,n$.

In conclusion, under Condition A, by taking the total spatial and time shifts $x^*$ and $t^*$ as in  (\ref{eqn:1.13}) and (\ref{eqn:1.b}), we see that all $I_1$, $I_2$ and $I_3$ of (\ref{eqn:2.6}) vanish.

\section{Proof of Theorems}
We begin by recalling an elementary lemma on the $L^p$ estimates for derivatives of the heat kernel. 
\begin{lemma}\label{lem4}
Let $1 \leq p \leq \infty$. Then, for all $k \in \N_0$ and all multi-indices $\al\in (\N_0)^n$, it holds that
\begin{align*}
\|\pt_t^k \pt_x^\al G_t(x) \|_{L^p(\R^n)} \leq C\,t^{-k-\frac{|\al|}{2}-\frac{n}{2}(1-\frac{1}{p})} \quad \mbox{for\,\,all}\,\, t>0.
\end{align*}
Here $C$ is a constant depending only on $p$, $k$ and $\al$.
\end{lemma}
\begin{proof}
We note that there exists a polynomial $P_{k,\al}(z)$ of $z$ such that $\pt_t^k \pt_x^\al G_t(x)$ is expressed as 
\begin{align*}
\pt_t^k \pt_x^\al G_t(x)=t^{-k-\frac{\,|\al|\,}{2}-\frac{\,n\,}{2}}P_{k,\al}\left(\frac{x}{\,t^{\frac{1}{2}}\,}\right)\e\left(-\frac{\,|x|^2\,}{\,4t\,}\right).
\end{align*}
Therefore, by the change of variables from $x/t^{1/2}$ to $z$, we find that there exists a constant
$C=C(p,k, \al)$ such that
\begin{align*}
\|\pt_t^k \pt_x^\al G_t(x)\|_{L^p(\R^n_x)}&=\left(\int_{\R^n}\,\left| t^{-k-\frac{\,|\al|\,}{2}-\frac{\,n\,}{2}}P_{k,\al}
\left(\frac{x}{\,t^{\frac{1}{2}}\,}\right) \e\left(-\frac{\,|x|^2\,}{4t}\right)\right|^p\,dx\right)^{\frac{\,1\,}{p}}   \\
                                                              &= t^{-(k+\frac{\,|\al|\,}{2}+\frac{n}{2})} \left(\int_{\R^n} |P_{k,\al}(z)|^p \e
\left(-\frac{\,p z^2\,}{4}\right)\,t^{\frac{\,n\,}{2}}\,dz\right)^{\frac{\,1\,}{p}}    \\
                                                              &\leq C\, t^{-k-\frac{\,|\al|\,}{2}-\frac{\,n\,}{2}(1-\frac{\,1\,}{p})} \quad \mbox{for\,\,all} \;\,  t > 0,
\end{align*}
which gives the desired estimate.
\end{proof}
\noindent $\bf{Proof \,of \,Theorem \,1.5.}$
We give the proof of Theorem 1.5, which includes that of Theorem 1.1 as the case when $k=0$. 
Returning to (\ref{eqn:2.6}), we have
\begin{multline}
\|u(x,t)-\Gk(x) \|_{L^p(\R^n_x)} \leq \|I_1+I_2+I_3\|_{L^p(\R^n_x)} +\|\int_{\R^n}\,R_{k+2} (G_t)(x,y) \,f(y) \,dy\|_{L^p(\R^n_x)}\\
+\|\hat{R}_k(x,t)\|_{L^p(\R^n_x)}. \n
\end{multline}
Since $x^{*,\al}$ and $t^{*,\al}$, $\al\in\Lambda_k(f)$, are given by (\ref{eqn:1.13}) and (\ref{eqn:1.b}), we see from Condition A  that $I_1=I_2=I_3=0$.  
Hence, it remains to show the time decay estimates for the terms $\|\int_{\R^n}\,R_{k+2} (G_t)(x,y) \,f(y) \,dy\|_{L^p(\R^n_x)}$ and $\|\hat{R}_k(x,t)\|_{L^p(\R^n_x)} \n$.

Utilizing the integral form of the Minkovski inequality, we obtain 
\begin{multline}\label{eqn:3.1.1}
\|\int_{\R^n}\,R_{k+2} (G_t)(x,y) \,f(y) \,dy\|_{L^p(\R^n_x)} \\ 
= \Biggl(\int_{\R^n_x} \biggl| \Bigl(\int_{\R^n_y}\bigl( (-1)^{k+2}(k+2)\int^1_0(1-\theta)^{k+1}\sum_{|\al|=k+2}\frac{1}{\al !}\bigl(\pt^\al_xG_t(x-y\theta) -\pt^\al_x G_t(x)\bigr)  y^\al d\theta f(y) \bigr) dy\Bigr) \biggr|^p dx\Biggr)^{\frac{1}{p}}  \\
\leq (k+2) \int^1_0 \biggl(\int_{\R^n_x} \Bigl( \int_{\R^n_y} \sum_{|\al|=k+2} |\pt^\al_xG_t(x-y\theta)-\pt^\al_x G_t(x) \, |y|^{k+2}|f(y)| dy \Bigr)^p dx \biggr)^{\frac{1}{p}} d\theta \hspace{30pt} \\
\leq (k+2)\int^1_0 \Biggl(\int_{\R^n_y} \biggl(\int_{\R^n_x}\Bigl(\sum_{|\al|=k+2} |\pt^\al_xG_t(x-y\theta)-\pt^\al_x G_t(x) ||y|^{k+2}|f(y)|  \Bigr)^p dx\biggr)^{\frac{1}{p}} dy\Biggr)  d\theta \hspace{40pt}  \\
\leq (k+2)\int^1_0 \int_{\R^n_y} \sum_{|\al|=k+2} \| \pt^\al_xG_t(x-y\theta)-\pt^\al_x G_t(x)  \|_{L^p(\R^n_x)}  |y|^{k+2}|f(y)| dy d\theta. \hspace{93pt} 
\end{multline}
Similar argument used in the proof of Lemma \ref{lem4} yields 
\begin{multline}\label{eqn:3.2.1}
\sum_{|\al|=k+2}\|\pt^\al_xG_t(x-y\theta)-\pt^\al_x G_t(x)\|_{L^p(\R^n_x)}  \\ 
= t^{-\frac{\,k+2\,}{2}-\frac{\,n\,}{2}}\sum_{|\al|=k+2}\Bigl(\int_{\R^n_x}\Bigl|P_{0,\al}\Bigl(\frac{\,x-y\theta\,}{t^{\frac{1}{2}}}\Bigr)\e\Bigl(-\frac{\,|x-y\theta|^2\,}{4t}\Bigr) 
 -P_{0,\al}\Bigl(\frac{x}{\,t^{\frac{1}{2}}\,}\Bigr)\e\Bigl(-\frac{\,|x|^2\,}{4t}\Bigr)\Bigr|^p dx\Bigr)^{\frac{1}{\,p}\,} \hspace{28pt} \\ 
= t^{-\frac{,k+2\,}{2}-\frac{\,n\,}{2}}\sum_{|\al|=k+2}\Bigl( \int_{\R^n_z} \Bigl|P_{0,\al}(z-t^{-\frac{\,1\,}{2}}y\theta)\e
\Bigl(-\frac{\,|z-t^{-\frac{1}{2}}y\theta|^2\,}{4}\Bigr)
-P_{0,\al}(z)\e\Bigl(-\frac{\,|z|^2\,}{4}\Bigr)\Bigr|^p t^{\frac{\,n\,}{2}} dz\Bigr)^{\frac{\,1\,}{p}}  \\ 
= t^{-\frac{\,k+2\,}{2}-\frac{\,n\,}{2}+\frac{n}{\,2p\,}}\sum_{|\al|=k+2}\Bigl( \int_{\R^n_z}\Bigl|
P_{0,\al}(z-t^{-\frac{\,1\,}{2}}y\theta)G_1(z-t^{-\frac{\,1\,}{2}}y\theta)
-P_{0,\al}(z)G_1(z) \Bigl|^p dz\Bigr)^{\frac{\,1\,}{p}} \hspace{50pt} \\ 
\leq C t^{-\frac{\,k+2\,}{2}-\frac{\,n\,}{2}(1-\frac{\,1\,}{p})}\sum_{|\al|=k+2}\Bigl( \int_{\R^n_z}\Bigl|\pt^\al_x G_1
(z-t^{-\frac{\,1\,}{2}}y\theta)-\pt^\al_x G_1(z)\Bigr|^p dz \Bigr)^{\frac{\,1\,}{p}}. \hspace{114pt} 
\end{multline}
It follows from (\ref{eqn:3.1.1}) and (\ref{eqn:3.2.1}) that
\begin{align*}
\phantom{}\|\int_{\R^n}R_{k+2}(G_t)(x,y)f(y)\,dy\|_{L^p(\R^n_x)}\leq \,C\,t^{-\frac{k+2}{2}-\frac{n}{2}(1-\frac{1}{p})}\int^1_0\left(\int_{\R^n_y}\varphi^{(k+2)}_t(y,\theta)|y|^{k+2}|f(y)|\,dy\right)\,d\theta,
\end{align*}
where $\varphi^{(k+2)}_t(y,\theta)$ is defined by 
\begin{align}\label{eqn:varphi}
\varphi^{(\ell)}_t(y,\theta) = \sum_{|\al|=\ell}\| \pt^\al_x G_1\left(\cdot-t^{-\frac{1}{2}}y\theta\right)-\pt^\al_x G_1(\cdot)\|_{L^p(\R^n)}, \quad \ell=1,2,\dots.
\end{align}
By the dominated convergence theorem, we can see that 
\begin{align*}
\lim_{t \to \infty}\left( \varphi^{(\ell)}_t(y,\theta)\right) = 0 \;\;\mbox{for fixed}\,\,y \in \R^n, \,\, 
\theta\in [0,\,1], \; \mbox{and} \,\,\ell=1,2,\dots.
\end{align*}
Since the functions $\varphi^{(k+2)}_t(y,\theta)$ is bounded in $t$, $\theta$, $y$, and $|y|^{k+2}|f(y)|$ is integrable on $\R^n$, 
using the dominated convergence theorem again, we have
\begin{align*}
\lim_{t \to \infty}\int^1_0\left(\int_{\R^n_y}\varphi^{(k+2)}_t(y,\theta)|y|^{k+2}|f(y)|\,dy\right)\,d\theta =0.
\end{align*}
Therefore, we conclude that 
\begin{align}\label{eqn:3.3.1}
\lim_{t \to \infty}t^{\frac{k+2}{2}+\frac{n}{2}(1-\frac{1}{p})}\|\int_{\R^n}R_{k+2}(G_t)(x,y)f(y)\,dy\, \|_{L^p(\R^n_x)} = 0.
\end{align}

We turn to the term $\|\hat{R}_k(x,t)\|_{L^p(\R^n_x)}$. It is observed that
\begin{align*}
\|\hat{R}_k(x,t)\|_{L^p(\R^n_x)} \leq J_1(t)+J_2(t)+J_3(t)+J_4(t)
\end{align*}
where
\begin{align*}
J_1(t)&=\|\sum_{\al\in\Lambda_k(f)}\frac{(-1)^k}{\,\alpha!\,}\M_{\alpha}(f)R_2(\partial^\alpha_x G_t)(x,x^{*,\alpha})\|_{L^p(\R^n_x)},  \\
J_2(t)&=\displaystyle \|\sum_{\al\in\Lambda_k(f)}\frac{(-1)^k}{\,\alpha!\,}\M_{\alpha}(f)\sum_{i=1} \pt_t \pt^{\al+e_i}_xG_t(x)x^{*,\al}_i\,t^{*,\al}\|_{L^p(\R^n_x)},\\
J_3(t)&=\displaystyle\|\sum_{\al\in\Lambda_k(f)}\frac{(-1)^k}{\,\alpha!\,}\M_{\alpha}(f)R_1(\partial_t\partial^\alpha_x G_t)(x,x^{*,\alpha})(-t^{*,\alpha})\|_{L^p(\R^n_x)}\hspace{130pt}\\
\mbox{and}\hspace{50pt}&\\
J_4(t)&=\displaystyle \|\sum_{\al\in\Lambda_k(f)}\frac{(-1)^k}{\,\alpha!\,}\M_{\alpha}(f)T_1(\partial^\alpha_xG_t)(x-x^{*,\alpha},t,t^{*,\alpha})\|_{L^p(\R^n_x)}.
\end{align*}
Since the time decay estimates for $J_1(t)$, $J_2(t)$ and $J_3(t)$ can be obtained directly by Lemma 3.1, we are going to give the 
estimate only for $J_4(t)$. It is easy to see that 
\begin{align*}
J_4(t)&=\|\sum_{\al\in\Lambda_k(f)}\frac{(-1)^k}{\,\alpha!\,}\M_{\alpha}(f)T_1(\partial^\alpha_xG_t)(x-x^{*,\alpha},t,t^{*,\alpha})\|_{L^p(\R^n_x)} \\
&= \|\sum_{\al\in\Lambda_k(f)}\frac{(-1)^k}{\,\alpha!\,}\M_{\alpha}(f) \Bigl((-1)\int^1_0(1-\theta)\Bigl(\bigl(\pt_t \pt _x^\al G_t \bigr)_{t-\theta t^{*,\al}}(x-x^{*,\al}) \\
&\hspace{240pt}-\pt_t \pt_x^\al G_t(x-x^{*,\al})\Bigr)(-t^{*,\al})\,d\theta \Bigr)\|_{L^p(\R^n_x)} \\
&\leq |\M_k|\,\,t^{*,k} \int^1_0 \sum_{\al\in\Lambda_k(f)} \|\left(\pt_t\pt_x^\al G_t\right)_{t-\theta t^{*,\al}}(x-x^{*,\al})-\pt_t \pt_x^\al G_t(x-x^{*,\al})\|_{L^p(\R^n_x)}\,d\theta,
\end{align*}
where $ |\M_k|=\max_{\al \in \Lambda_k(f)}|\M_\al(f)|$ and $t^{*,k}=\max\,\{\,t^{*,\alpha}, 0\,\,|\, \alpha\in \Lambda_k(f)\}$.
By the similar way in deriving (\ref{eqn:3.2.1}), we have 
\begin{align*}
&\sum_{\al\in\Lambda_k(f)}\|\left(\pt_t \pt_x^\al G_t\right)_{t-\theta t^{*,\al}}(x-x^{*,\al})-\pt_t \pt_x^\al G_t(x-x^{*,\al})\|_{L^p(\R^n_x)} \\
&= \sum_{\al\in\Lambda_k(f)} \Biggl(\int_{\R^n_x}\biggl|(t-\theta t^{*,\al})^{-\frac{k+2}{2}-\frac{n}{2}}P_{1,\al}\left(\frac{x}{(t-\theta t^{*,\al})^{\frac{1}{2}}}\right)\e\left(-\frac{|x|^2}{4(t-\theta t^{*,\al})}\right) \hspace{300pt} \\
&\hspace{210pt} -t^{-\frac{k+2}{2}-\frac{n}{2}}P_{1,\al}\left(\frac{x}{t^{\frac{1}{2}}}\right)\e\left(-\frac{|x|^2}{4t}\right)\biggr|^pdx\Biggr)^{\frac{1}{p}}  \\
&\leq t^{-\frac{k+2}{2}-\frac{n}{2}(1-\frac{1}{p})}\sum_{\al\in\Lambda_k(f)}\|\biggl(1-\theta \frac{t^{*,\al}}{t}\biggr)^{-\frac{n+2}{2}}\pt_t\pt_x^\al G_1\biggl(z\Bigl(1-\theta\frac{t^{*,\al}}{t}\Bigr)^{-\frac{1}{2}}\biggr)-\pt_t\pt_x^\al G_1(z)\|_{L^p(\R^n_z)}.
\end{align*}
On the other hand, it is not difficult to see that  
$$
\lim_{t \to \infty}\biggl(\sum_{\al\in\Lambda_k(f)}\|\Bigl(1-\theta \frac{t^{*,\al}}{t}\Bigr)^{-\frac{n+2}{2}}\pt_t\pt_x^\al G_1\biggl(
z\Bigl(1-\theta\frac{t^{*,\al}}{t}\Bigr)^{-\frac{1}{2}}\biggr)-\pt_t\pt_x^\al G_1(z)\|_{L^p(\R^n_z)}\biggr) = 0
$$
for fixed $\theta\in [0,1]$. Thus, we reach 
\begin{align*}
\lim_{t \to \infty}t^{\frac{k+2}{2}+\frac{n}{2}(1-\frac{1}{p})} J_4(t)= 0,
\end{align*}
so that 
\begin{align}\label{eqn:3.4.1}
\lim_{t \to \infty}t^{\frac{k+2}{2}+\frac{n}{2}(1-\frac{1}{p})} \|\hat{R}_k(x,t)\|_{L^p(\R^n_x)} = 0.
\end{align}
Now we can conclude from (\ref{eqn:3.3.1}), (\ref{eqn:3.4.1}) that
\begin{align*}
\lim_{t \to \infty}t^{\frac{k+2}{2}+\frac{n}{2}(1-\frac{1}{p})}\|u(\cdot,t)-\Gk(x)\|_{L^p(\R^n_x)} = 0,
\end{align*}
which is the desired estimate (\ref{eqn:1.16}) of Theorem 1.5. 

Finally, we consider the case when $\int_{\R^n} (1+|y|)^{k+3}|f(y)|\,dy < +\infty$. In this case, we have 
\begin{align*}
\|\sum_{|\al|=k+2}\pt^\al_xG_t(x-y\theta)-\pt^\al_x G_t(x)\|_{L^p(\R^n_x)}  
&=\|\sum_{|\al|=k+2}\pt^\al_x \int^1_0 \sum_{|\beta|=1} \pt^\beta_x G_t(x-y\theta \phi)(-y)^\beta \,d\phi \|_{L^p(\R^n_x)} \\
&\leq |y|\int^1_0\sum_{|\al|=k+3} \| \pt^\al_x G_t(x-y\theta \phi) \|_{L^p(\R^n_x)}\,d\phi  \\
&\leq |y|\,\,t^{-\frac{k+3}{2}-\frac{n}{2}(1-\frac{1}{p})}\sum_{|\al|=k+3} \| \pt^\al_xG_1(z) \|_{L^p(\R^n_z)}  ,
\end{align*}
thereby  
\begin{align*}
\|\int_{\R^n}\,R_{k+2} (G_t)(x,y) \,f(y) \,dy\|_{L^p(\R^n_x)} \leq C\int_{\R^n}|y|^{k+3}|f(y)|\,dy \cdot t^{-\frac{k+3}{2}-\frac{n}{2}(1-\frac{1}{p})}.
\end{align*}
While $J_4(t)$ can be estimated as 
\begin{align*}
J_4(t)\leq C\,t^{-\frac{k+4}{2}-\frac{n}{2}(1-\frac{1}{p})}\;\;\mbox{for}\;t>t^{*,k},
\end{align*}
because 
\begin{multline*}
\sum_{\al\in\Lambda_k(f)}\|\left(\pt_t \pt_x^\al G_t\right)_{t-\theta t^{*,\al}}(x-x^{*,\al})-\pt_t \pt_x^\al G_t(x-x^{*,\al})\|_{L^p(\R^n_x)}\\
\leq \int_0^1\sum_{\al\in\Lambda_k(f)} \|\pt_t^2 \pt_x^\al G_{t-t^{*,\al}\theta \phi}(x) (-t^{*,\al}\theta) \|_{L^p(\R^n_x)}d\phi 
\end{multline*}
\begin{align*}
&\leq t^{*,k} \, t^{-\frac{k+4}{2}-\frac{n}{2}(1-\frac{1}{p})}\int_0^1\sum_{\al\in\Lambda_k(f)}\|\Bigl(1-\theta\phi \frac{t^{*,\al}}{t}\Bigr)^{-\frac{k+n+4}{2}}\\
&\hspace{100pt} \times P_{2,\al}\Bigl(z \bigl(1-\theta \phi \frac{t^{*,\al}}{t}\bigr)^{-\frac{1}{2}}\Bigr)\e\Bigl(-\frac{|z (1-\theta\phi\frac{t^{*,\al}}{t})^{-\frac{1}{2}}|^2}{4}\Bigr)\|_{L^p(\R^n_z)}d\phi\\
&\leq C\,t^{-\frac{k+4}{2}-\frac{n}{2}(1-\frac{1}{p})} \;\; \mbox{for} \; t>t^{*,k}.
\end{align*}
Combining the decay estimates for $J_1(t)$, $J_2(t)$ and $J_3(t)$ corresponding to this case, we reach the estimate
\begin{align*}
\|u(\cdot,t)-\Gk(\cdot)\|_{L^p(\R^n_x)} \leq C\,\,t^{-\frac{k+3}{2}-\frac{n}{2}(1-\frac{1}{p})} \;\; \mbox{for} \; t>t^{*,k},
\end{align*}
which is just the estimate  (\ref{eqn:1.9}) in Remark 1.2 (i). 

\medskip

$\bf{Proof \,of \,Theorem \,1.10.}$
We next give the proof of Theorem 1.10, in which that of Theorem 1.3 is included as the case when $k=0$. 
The following two elementary lemmas play an important role in the proof of Theorem 1.10. The first lemma is 
\begin{lemma}\label{lemma3}%Lemma 3.1%
Let $\alpha=(\alpha_1,\dots,\alpha_n)\in (\N_0)^n$ and let $\omega=(\omega_1,\dots,\omega_n)\in S^{n-1}$. If at least 
one component of 
$\alpha$ is odd, then 
$$
\int_{S^{n-1}}\,\omega^\alpha d\omega=0
$$
holds. 
On the other hand, if all components of $\alpha$ are even, that is, $\alpha=(2m_1,\dots,2m_n)$ with $m_i\in \N_0, \,\, i=1,\dots,n$, then  
\begin{eqnarray}
\displaystyle 
 \int_{S^{n-1}}\,\omega^{2m} d\omega&=& 2\,\frac{\,\,\prod^n_{i=1}\,\Gamma(m_i+\frac{\,1\,}{2})\,\,}{\,\Gamma(|m|+\frac{\,n\,}{2})\,}\nonumber\\
%&=& \displaystyle \frac{\,\prod^n_{i=1}\left(\prod^{[m_i]_{\geq 1}}_{j=1}\,[m_i+\frac{\,1\,}{2}-j]_{\geq 1}\right)\,}
%{\,\prod^{|m|}_{i=1}(|m|+\frac{\,n\,}{2}-i)} |S^{n-1}|
 \nonumber
\end{eqnarray}
holds. Here $\omega^\alpha={\omega_1}^{\alpha_1}\dots {\omega_n}^{\alpha_n}$, 
$|m|=\sum^n_{i=1} m_i$, and  $[k]_{\geq 1}=\max \{k,\,1\}$ for an integer $k$.
\end{lemma}
\noindent The second lemma concerns an explicit form of the spatial derivatives of the heat kernel. Since $G_t(x)$ 
depends on $x$ through $r=|x|$, we use the notation that $G_t(x)=\tilde{G}_t(r)$.
\begin{lemma}\label{lemma4}%Lemma 3.1%
For the same $\alpha$ and $\omega$ as in the above lemma, it holds that 
\begin{equation}\label{eqn:3.2}\displaystyle 
\partial_x^\alpha G_t(x)=\left(-\frac{r}{\,2t\,}\right)^{|\alpha|}\tilde{G}_t(r)\,\omega^\alpha+
\sum_{\beta\in K(\alpha)\setminus\{0\}} C_{\alpha, \beta} \,\omega^{\alpha(\beta)}\\,\,\mbox{for \, all}\,\,r=|x|>0,\,t>0, 
\end{equation}
where $\beta \in(\N_0)^n$ are multi-indices and the set $K(\alpha)$ is defined by 
$$
K(\alpha)=\{\,\beta=(\beta_1,\dots,\beta_n)\in(\N_0)^n\,|\, 0\leq \beta_i \leq [\alpha_i/2],\,i=1,\dots,n\,\}
$$ 
and $\alpha(\beta)=(\alpha_1-2\beta_1,\dots, \alpha_n-2\beta_n).$  The coefficient $C_{\alpha,\beta}$ consists of the term  
$\displaystyle \left(-\frac{r}{\,2t\,}\right)^{|\alpha|-|\beta|}\frac{1}{\,r^{|\beta|}}\tilde{G}_t(r)$ multiplied by a constant depending only on 
$\alpha$ and $\beta$.
\end{lemma}
\noindent The proof of Lemmas 3.2 and 3.3 is given in Appendix A.

Instead of (\ref{eqn:2.1}), we use here the following higher order expansion of the heat solution $u(x,t)$ 
given by  (\ref{eqn:1.3}): 
\begin{multline}\label{eqn:3.7.1}
u(x,t)= \int_{\R^n} G_t(x-y) f(y)\,dy  \\
    	= \sum^{k+3}_{i=0} \frac{1}{\,i!\,}\int_{\R^n}\, L_n(D)^i\,G_t(x)\,f(y)\,dy+\int_{\R^n}\,R_{k+3} (G_t)(x,y)\,f(y) \,dy \hspace{20pt} \\
	=\sum_{|\alpha|\leq k-1}\frac{\,(-1)^{|\alpha|}\,}{\alpha!}\M_\alpha(f)\partial^\alpha_x G_t(x) +\sum_{\alpha \in \Lambda_k(f) }\frac{\,(-1)^{|\alpha|}\,}{\alpha!}\M_\alpha(f)\partial^\alpha_x G_t(x) 	\hspace{10pt} \\
+\frac{1}{\,(k+1)!\,}\int_{\R^n}\, L_n(D)^{k+1}\,G_t(x)\,f(y)\,dy+\frac{1}{\,(k+2)!\,}\int_{\R^n}\, L_n(D)^{k+2}\,G_t(x)\,f(y)\,dy \\
+\sum_{|\alpha|=k+3} \frac{\,(-1)^{|\alpha|}\,}{\alpha!}\M_\alpha(f)\partial^\alpha_x G_t(x)+\int_{\R^n}\,R_{k+3} (G_t)(x,y) \,f(y) \,dy.
\end{multline}
Similarly, instead of  (\ref{eqn:2.4}), we employ an expansion of the $k$-th order modified heat kernel $\Gk(x)$ in 
(\ref{eqn:1.12}) such that
\begin{multline}\label{eqn:3.8.1}
\Gk(x)=\Gk(x; x^*,t^*)  \\ 
	=\displaystyle \sum_{0\leq|\alpha|\leq k-1}\frac{\,(-1)^{|\alpha|\,}}{\alpha!} \M_\alpha(f)\partial_x^\alpha G_t(x)
	+\sum_{\alpha \in \Lambda_k(f)}\frac{\,(-1)^{|\alpha|\,}}{\alpha!} \M_\alpha(f)\partial_x^\alpha G_t(x) \hspace{70pt} \\
-\displaystyle \sum_{\alpha \in \Lambda_k(f)}\sum^n_{i=1}\frac{\,(-1)^k\,}{\alpha!} \M_\alpha(f) x^{*,\alpha}_i
\partial_x^{\alpha+e_i} G_t(x)  
+\displaystyle \sum_{\alpha \in \Lambda_k(f)}\sum_{|\alpha'|
=2}\frac{(-1)^k}{\alpha!}\frac{\,1\,}{\alpha' !}\, \M_{\alpha}(f) (x^{*,\alpha})^{\alpha'} \partial_x^{\alpha+\alpha'}G_t(x) \\
-\displaystyle 
\sum_{\alpha \in \Lambda_k(f)}\sum_{i=1}^n \frac{(-1)^k}{\alpha!}\M_\alpha(f) t^{*,\alpha} \,\pt^2_{x_i}\pt^{\al}_x G_t(x)
+ \sum_{\alpha \in \Lambda_k(f)}\sum_{i=1}^n\sum_{j=1}^n\frac{(-1)^k}{\alpha!}
\M_\alpha(f)\,x_j^{*,\al}t^{*,\al}\pt^2_{x_i}\pt_x^{\al+e_j}G_t(x)+\hat{R}'_k(x,t). 
\end{multline}
Here
\begin{equation}\label{eqn:3.9.1}
\hat{R}'_k(x,t)=\displaystyle \sum_{\al \in \Lambda_k(f)}\frac{(-1)^k}{\,\alpha!\,}\M_{\alpha}(f)\tilde{R}'_\alpha(x,t),
\end{equation}
where
\begin{multline*}
\tilde{R}'_\alpha(x,t)=R_2(\partial^\alpha_x G_t)(x,x^{*,\alpha})+R_1(\partial_t\partial^\alpha_x G_t)(x,x^{*,\alpha})(-t^{*,\alpha})+\pt_t^2\pt^\al_xG_t(x)(t^{*,\al})^2 \\ 
- \sum_{i=1}^n \pt_t^2 \pt^{\al+e_i}_xG_t(x)x^{*,\al}_i(t^{*,\al})^2+ R_1\left(\pt_t^2\pt^\al_xG_t\right)(x,x^{*,\al})(t^{*,\al})^2+T_2(\partial^\alpha_xG_t)(x-x^{*,\alpha},t,t^{*,\alpha}).
\end{multline*}
Taking the polar coordinates $x=r\omega$ where $r=|x|$, $\displaystyle \omega=\frac{x}{\,|x|\,}\in S^{n-1}$, for all 
$x\in \R^n\setminus \{0\}$, we can see from (\ref{eqn:3.7.1}), (\ref{eqn:3.8.1}) that the integration over $S^{n-1}$ of the error 
term of the heat solution $u(x,t)$ for $\Gk(x)$ is given by
\begin{multline}\label{eqn:3.10.1}
\int_{S^{n-1}}u(r\omega,t)\, d\omega-\int_{S^{n-1}}\Gk(r\omega)\, d\omega \\
=\sum_{i=1}^4\int_{S^{n-1}} I_i\, d\omega +\int_{S^{n-1}}\Bigg\{ \int_{\R^n}\,R_{k+3} (G_t)(r\omega,y) \,f(y) \,dy-\hat{R}'_k(r\omega,t)\Bigg\}d\omega \,\,\,\mbox{for\,all}\,$t>0$,\, $r>0$,
\end{multline}
where the terms $I_1$, $I_2$, $I_3$ are the same as (\ref{eqn:2.7}), (\ref{eqn:2.8}), (\ref{eqn:2.9}), and 
\begin{multline*}
I_4=\sum_{|\alpha|=k+3} \frac{(-1)^{k+3}}{\alpha!}\M_{\alpha}(f) \partial_x^{\alpha}G_t(x)  - 
\sum_{\alpha \in \Lambda_k(f)}\,
\sum_{i=1,j}^n\,\frac{(-1)^k}{\alpha!}\M_{\alpha}(f)\,x_j^{*,\alpha}\,\pt_{x_i}^2\partial_x^{\alpha+e_j}G_t(x). 
\end{multline*}

We first have a look at the case when $k$ is even. 
Since the initial data $f$ satisfies (\ref{eqn:1.a}) in Condition A and the time shifts $t^{*,\alpha}$, $\alpha\in\Lambda_k(f)$, are 
given by (\ref{eqn:1.b}), we find that $I_2=0$. 
On the other hand,  remembering the assumption that 
$\M_{\alpha+\alpha'}(f)=0$ for all $\alpha\notin \Lambda_k(f)$ with $|\alpha|=k$ and all $\alpha'\in(\N_0)^n$ with $|\alpha'|=2$ 
in (\ref{eqn:1.c}), thanks to Lemmas \ref{lemma3} and \ref{lemma4}, we can see that the terms $I_1$, $I_3$ and $I_4$ obey 
$$
\displaystyle \int_{S^{n-1}}\,I_1\,d\omega=\int_{S^{n-1}}\, I_3 \,d\omega=\int_{S^{n-1}}\, I_4 \,d\omega=0.
$$
In particular, we notice the fact that $\int_{S^{n-1}}\,I_1\,d\omega=0$ holds for arbitrary spatial shifts $x^{*,\alpha}$, 
$\alpha\in\Lambda_k(f)$, when $k$ is even. In fact, since all derivatives of $G_t(x)$ with respect to $x$ appearing in $I_1$ are 
those of odd orders when $k$ is even, Lemmas 3.2 and 3.3  readily yield the fact above. 
Consequently, we have
\begin{multline}\label{eq:3.11.1}
\|[u(\cdot, t))]-[\Gk(\cdot;x^*,t^*)]\,\|_{L^p(\R^n)}\\
\hspace{95pt}=\|\int_{S^{n-1}}\left\{ \int_{\R^n}\,R_{k+3}(G_t)(r\omega,y)\,f(y)\,dy-{\hat{R}_k}'(r\omega,t) \right\} \,d\omega\|_{L^p(\R^n)} \\
\leq \bigg\{\|  \int_{\R^n}\,R_{k+3}(G_t)(x,y)\,f(y)\,dy\|_{L^p(\R^n)}+\|{\hat{R}_k}'(x,t) \|_{L^p(\R^n)}\bigg\}|S^{n-1}|.
\end{multline}
By the same way as in (\ref{eqn:3.1.1}), (\ref{eqn:3.2.1}),  we obtain 
\begin{multline}\label{eqn:3.12.1}
\|\int_{\R^n}\,R_{k+3} (G_t)(x,y) \,f(y) \,dy\|_{L^p(\R^n)}  \\
\leq (k+3)\int^1_0 \int_{\R^n} \sum_{|\al|=k+3} \| \pt^\al_xG_t(x-y\theta)-\pt^\al_x G_t(x)  \|_{L^p(\R^n_x)}\, |y|^{k+3}|f(y)| \, dy d\theta, \\
\end{multline}
and 
\begin{multline}\label{eqn:3.13.1}
\sum_{|\al|=k+3}\|\pt^\al_xG_t(x-y\theta)-\pt^\al_x G_t(x)\|_{L^p(\R^n_x)}  \\
= t^{-\frac{k+3}{2}-\frac{n}{2}}\sum_{|\al|=3}\Biggl(\int_{\R^n_x}\biggl|P_{0,\al}\biggl(\frac{x-y\theta}{t^{\frac{1}{2}}}\biggr)\e\biggl(-\frac{|x-y\theta|^2}{4t}\biggr) 
-P_{0,\al}\biggl(\frac{x}{t^{\frac{1}{2}}}\biggr)\e\biggl(-\frac{|x|^2}{4t}\biggr)\biggr|^p dx\Biggr)^{\frac{1}{p}} \hspace{30pt} \\
= t^{-\frac{k+3}{2}-\frac{n}{2}}\sum_{|\al|=k+3}\Biggl( \int_{\R^n_z} \biggl|P_{0,\al}\biggl(z-t^{-\frac{1}{2}}y\theta\biggr)\e\biggl(-\frac{|z-t^{-\frac{1}{2}}y\theta|^2}{4}\biggr) -P_{0,\al}(z)\e\biggl(-\frac{|z|^2}{4}\biggr)\biggr|^p t^{\frac{n}{2}} dz\Biggr)^{\frac{1}{p}}  \\
\leq C t^{-\frac{k+3}{2}-\frac{n}{2}(1-\frac{1}{p})}\sum_{|\al|=k+3}\|\pt^\al_x G_1\left(z-t^{-\frac{1}{2}}y\theta\right)-\pt^\al_x G_1(z)\|_{L^p(\R^n_z)}. \hspace{150pt}
\end{multline}
It follows from (\ref{eqn:3.12.1}), (\ref{eqn:3.13.1}) that
\begin{align*}
\|\int_{\R^n}R_{k+3}(G_t)(x,y)f(y)\,dy\|_{L^p(\R^n)} \leq C\,t^{-\frac{k+3}{2}-\frac{n}{2}(1-\frac{1}{p})}\int^1_0\left(\int_{\R^n}\varphi^{(k+3)}_t(y,\theta)|y|^{k+3}|f(y)|\,dy\right)\,d\theta. 
\end{align*}
where $\varphi^{(k+3)}_t(y,\theta)$ is the function defined by (\ref{eqn:varphi}) with $\ell=k+3$.
Since the functions $\varphi^{(k+3)}_t(y,\theta)$ is bounded in $t$ , $\theta$ and $y$, and $|y|^{k+3}|f(y)|$ is integrable on $\R^n$, 
it is seen that
\begin{align*}
\lim_{t \to \infty}\int^1_0\left(\int_{\R^n}\varphi^{(k+3)}_t(y,\theta)|y|^{k+3}|f(y)|\,dy\right)\,d\theta =0,
\end{align*}
thereby
\begin{align}\label{eqn:3.14.1}
\lim_{t \to \infty}t^{\frac{k+3}{2}+\frac{n}{2}(1-\frac{1}{p})}\|\int_{\R^n}R_{k+3}(G_t)(x,y)f(y)\,dy\, \|_{L^p(\R^n)} = 0.
\end{align}
We next estimate the term $\|\hat{R'_k}(x,t)\|_{L^p(\R^n)}$ . In view of (\ref{eqn:3.9.1}), we have
\begin{align}\label{eqn:3.16.1}
\|\hat{R'_k}(x,t)\|_{L^p(\R^n)} \leq \tilde{J_1}(t)+\tilde{J_2}(t)+\tilde{J_3}(t)+\tilde{J_4}(t)+\tilde{J_5}(t)+\tilde{J_6}(t),
\end{align}
where
\begin{align*}
\tilde{J_1}(t)&=\|\sum_{\al \in \Lambda_k(f)}\frac{(-1)^k}{\,\alpha!\,}\M_{\alpha}(f)R_2(\partial^\alpha_x G_t)(x,x^{*,\alpha})\|_{L^p(\R^n)} ,\\
\tilde{J_2}(t)&=\|\sum_{\al \in \Lambda_k(f)}\frac{(-1)^k}{\,\alpha!\,}\M_{\alpha}(f)t^{*,\alpha}R_1(\partial_t\partial^\alpha_x G_t)(x,x^{*,\alpha})\|_{L^p(\R^n)} ,\\
\tilde{J_3}(t)&=\|\sum_{\al \in \Lambda_k(f)}\frac{(-1)^k}{\,\alpha!\,}\M_{\alpha}(f)(t^{*,\al})^2\pt_t^2\pt^\al_xG_t(x)\|_{L^p(\R^n)} ,\\
\tilde{J_4}(t)&=\|\sum_{\al \in \Lambda_k(f)}\frac{(-1)^k}{\,\alpha!\,}\M_{\alpha}(f)\,x^{*,\al}_i(t^{*,\al})^2\sum_{i=1}^n \pt_t^2 \pt^{\al+e_i}_xG_t(x)\|_{L^p(\R^n)} , \hspace{120pt} \\
\tilde{J_5}(t)&= \|\sum_{\al \in \Lambda_k(f)}\frac{(-1)^k}{\,\alpha!\,}\M_{\alpha}(f)(t^{*,\al})^2R_1\left(\pt_t^2\pt^\al_xG_t\right)(x,x^{*,\al})\|_{L^p(\R^n)}, \\
\mbox{and}\hspace{80pt}&\\
\tilde{J_6}(t)&=\|\sum_{\al \in \Lambda_k(f)}\frac{(-1)^k}{\,\alpha!\,}\M_{\alpha}(f)T_2(\partial^\alpha_xG_t)(x-x^{*,\alpha},t,t^{*,\alpha})\|_{L^p(\R^n)} .
\end{align*}
Since the time decay estimates for $\tilde{J_1}(t)$, $\tilde{J_2}(t)$, $\tilde{J_3}(t)$, $\tilde{J_4}(t)$ and $\tilde{J_5}(t)$ can be 
shown by applying Lemma 3.1 directly, we here give only the estimate for $\tilde{J_6}(t)$.
In view of the definition $T_2(\partial^\alpha_x G_t)$ appearing below (\ref{eqn:2.5.1}), we have
\begin{multline*}
\tilde{J_6}(t)=\|\sum_{\al \in \Lambda_k(f)}\frac{(-1)^k}{\,\alpha!\,}\M_{\alpha}(f)T_2(\partial^\alpha_xG_t)(x-x^{*,\alpha},t,t^{*,\alpha})\|_{L^p(\R^n)} \\
\leq 2\,|\M_k||t^{*,k}|^2\,\int^1_0\sum_{\al \in \Lambda_k(f)}\|\Bigl(\pt_t^2\pt_x^\al G_t\Bigr)_{t- t^{*,\al}\theta }(x-x^{*,\al})-\pt_t^2\pt_x^\al G_t(x-x^{*,\al})\|_{L^p(\R^n_x)}\,d\theta.
\end{multline*}
Since
\begin{multline*}
\sum_{\al \in \Lambda_k(f)}\|\Bigl(\pt_t^2\pt_x^\al G_t\Bigr)_{t- t^{*,\al}\theta }(x-x^{*,\al})-\pt_t^2\pt_x^\al G_t(x-x^{*,\al})\|_{L^p(\R^n_x)}\\
= t^{-\frac{k+4}{2}-\frac{n}{2}}\sum_{\al \in \Lambda_k(f)}\Bigg(\int_{\R^n_z}\Bigg|\left(1-\frac{t^{*,\al}}{t}\,\theta \right)^{-\frac{n+k+4}{2}}P_{2,\al}\left(z\left(1-\frac{t^{*,\al}}{t}\,\theta \right)^{-\frac{1}{2}}\right)\e\left(-\frac{|z\left(1-\frac{t^{*,\al}}{t}\,\theta \right)^{-\frac{1}{2}}|^2}{4}\right) \\
-P_{2,\al}(z)\e\left(-\frac{|z|^2}{4}\right)\Bigg|^p t^{\frac{n}{2}}dz\Bigg)^{\frac{1}{p}} \\
\leq t^{-\frac{k+4}{2}-\frac{n}{2}(1-\frac{1}{p})}\sum_{\al \in \Lambda_k(f)}\|\left(1-\frac{t^{*,\al}}{t}\,\theta \right)^{-\frac{n+k+4}{2}}\pt_t\pt_x^\al G_1\left(z\left(1-\frac{t^{*,\al}}{t}\,\theta \right)^{-\frac{1}{2}}\right)
-\pt_t\pt_x^\al G_1(z)\|_{L^p(\R^n_z)},
\end{multline*}
and since
\begin{align*}
\lim_{t \to \infty}\left(\|\left(1-\frac{t^{*,\al}}{t}\,\theta \right)^{-\frac{n+k+4}{2}}\pt_t\pt_x^\al G_1
\left(z \left(1-\frac{t^{*,\al}}{t}\,\theta \right)^{-\frac{1}{2}}\right)-\pt_t\pt_x^\al G_1(z)\|_{L^p(\R^n_z)}\right) = 0
\end{align*}
for fixed $\theta \in[0,1]$, $\alpha\in\Lambda_k(f)$, we obtain
\begin{align*}
\lim_{t \to \infty}t^{\frac{k+4}{2}+\frac{n}{2}(1-\frac{1}{p})}\tilde{J_6}(t) = 0,
\end{align*}
so that
\begin{align}\label{eqn:3.17.1}
\lim_{t \to \infty}t^{\frac{k+4}{2}+\frac{n}{2}(1-\frac{1}{p})}\|\hat{R'_k}(x,t)\|_{L^p(\R^n)} = 0.
\end{align}
Therefore, combining (\ref{eq:3.11.1}), (\ref{eqn:3.14.1}) and (\ref{eqn:3.17.1}), we can conclude that 
\begin{equation*}
\lim_{t \to \infty}t^{\frac{k+3}{2}+\frac{n}{2}(1-\frac{1}{p})}\|[u(\cdot, t))]-[\Gk(\cdot;x^*,t^*)]\,\|_{L^p(\R^n)}= 0
\end{equation*}
Taking $x^{*,\alpha}=0$ for all $\alpha\in\Lambda_k(f)$ specifically, we have the decay estimate (\ref{eqn:1.26}) of Theorem 1.10. 

\medskip

We turn to the case when $k$ is odd.
Since the spatial shifts $x^{*,\alpha}$, for all $\alpha\in\Lambda_k(f)$, are taken as in (\ref{eqn:1.13}), we can 
see that $I_1= 0$.  Whereas,  remembering the assumption that 
$\M_{\alpha+e_i}(f)=0$, for all $\alpha\notin \Lambda_k(f)$ with $|\alpha|=k$ and all $i=1,\dots,n$, by virtue of 
Lemmas \ref{lemma3} and \ref{lemma4}, we find that the terms $I_2$ in (\ref{eqn:2.8}) and $I_3$ in (\ref{eqn:2.9}) obey 
$$
\displaystyle \int_{S^{n-1}}\,I_2\,d\omega=\int_{S^{n-1}}\, I_3 \,d\omega=0.
$$
We notice that Lemmas 3.2 and 3.3  lead to $\int_{S^{n-1}}\,I_2\,d\omega=0$ for arbitrary time shifts 
$t^{*,\alpha}$, $\alpha\in\Lambda_k(f)$, when $k$ is odd. 
Therefore, returning to (\ref{eqn:2.6}), we have
\begin{multline*}
\displaystyle 
\int_{S^{n-1}}\,u(r\omega,t)\,d\omega-\int_{S^{n-1}}\,\hat{G}_t^{(k)}(r\omega)\,d\omega
=\int_{S^{n-1}}\Bigg\{\int_{\R^n}\,R_{k+2}(G_t)(r\omega,y)\,f(y)\,dy-\hat{R}_k(r\omega,t)\Bigg\}d\omega,
\end{multline*}
so that 
\begin{align*}
\|[u(\cdot,t)-[\Gk(\cdot;x^*,t^*)]\,\|_{L^p(\R^n)}
&=\|\int_{S^{n-1}}\Bigg\{\int_{\R^n}\,R_{k+2}(G_t)(r\omega,y)\,f(y)\,dy-\hat{R}_k(r\omega,t)\Bigg\}\,d\omega\|_{L^p(\R^n)}\\
&\leq\bigg\{  \|\int_{\R^n}\,R_{k+2}(G_t)(x,y)\,f(y)\,dy\|_{L^p(\R^n)} +\|\hat{R}_k(x,t)\|_{L^p(\R^n)}  \bigg\}|S^{n-1}|.
\end{align*}
Since the time decay estimates for $\|\int_{\R^n}\,R_{k+2}(G_t)(x,y)\,f(y)\,dy\|_{L^p(\R^n_x)}$, $\|\hat{R}_k(x,t)\|_{L^p(\R^n_x)}$ 
have been already given in the proof of Theorem 1.5, we conclude that
$$
\lim_{t \to \infty} t^{\frac{k+2}{2}+\frac{n}{2}(1-\frac{1}{p})}\|[u(\cdot)]-[\Gk(x; x^*,t^*)]\,\|_{L^p(\R^n)}=0.
$$
In the estimate (\ref{eqn:1.27}) in Theorem 1.10, the total time shift $t^{*}$ is taken as zero. 
We complete the proof of Theorem 1.10. 

\section{Appendices}

$\bf{Appendix\, A.\,\, Proof \,of \,Lemmas\,\, 3.2\, and\,\, 3.3.}$

We first give the proof of Lemma 3.2. 
By the polar coordinates $r, \theta_1,\dots, \theta_{n-1}$, the 
Cartesian coordinates $x_1,\dots,x_n$ in $\R^n$ are given by 
\begin{eqnarray*}
x_1&=&r\cos\theta_1,\\
x_2&=&r\sin\theta_1\cos\theta_2,\\
\phantom{x-x}&\vdots&\phantom{x-x}\\
x_{n-1}&=&r\sin\theta_1\dots \sin\theta_{n-2}\cos\theta_{n-1},\\
x_{n}&=&r\sin\theta_1\dots \sin\theta_{n-2}\sin\theta_{n-1},\\
\end{eqnarray*}
with $r>0$ and $\,0\leq \theta_1,\dots,\theta_{n-2}\leq \pi,\,0\leq \theta_{n-1}< 2\pi$. 
We first consider the case when all components of $\alpha$ are even, that is, $\alpha=(2m_1,
\dots, 2m_n)$, $m_i\in\N$, $i=1,\dots,n$. Since $\displaystyle \omega_i=\frac{\,x_i\,}{r}$, $i=1,\dots,n$, we have
\begin{multline}\label{eqn:1.17}
\int_{S^{n-1}}\omega^\alpha\,d\omega=\int^{2\pi}_0\int^{\pi}_0\cdots\int^{\pi}_0
\omega_1^{2m_1}\cdots \omega_n^{2m_n}\prod^{n-2}_{i=1}\,\sin^{n-1-i}\theta_i\,\,d\theta_1\cdots \theta_{n-1}\\
=\int^{2\pi}\,\sin^{2m_n}\theta_{n-1}\cos^{2m_{n-1}}\theta_{n-1}\,d\theta_{n-1}\int^\pi_0\sin^{2m_n+2m_{n-1}+1}\theta_{n-2}\cos^{2m_{n-2}}\theta_{n-2}
\,d\theta_{n-2}\\
\times \int^\pi_0\sin^{2m_n+2m_{n-1}+2m_{n-2}+2}\theta_{n-3}
\cos^{2m_{n-3}}\theta_{n-3}\,d\theta_{n-2}\cdots\int^\pi_0\sin^{2m_n+\cdots+2m_2+(n-2)}
\theta_1\cos^{2m_1}\theta_1\,d\theta_1.\\
\end{multline}
Since for $p\in\N$ and $q$ even, 
\begin{eqnarray*}\displaystyle 
\int^\pi_0 \sin^p x\cos^q x\,dx&=&2\int^{\frac{\,\pi\,}{2}}_0\sin^p x\cos^q x\,dx\\
&=&\displaystyle B(\frac{\,p+1\,}{2},\frac{\,q+1\,}{2})=\frac{\,\Gamma(\frac{\,p+1\,}{2})\,\Gamma(\frac{\,q+1\,}{2})\,}{\Gamma(\frac{\,p+q\,}{2}+1)},
\end{eqnarray*}
and for both $p$ and $q$ even, 
\begin{eqnarray*}\displaystyle 
\int^\pi_0 \sin^p x\cos^q x\,dx&=&4\int^{\frac{\,\pi\,}{2}}_0\sin^p x\cos^q x\,dx\\
&=&\displaystyle 2B(\frac{\,p+1\,}{2},\frac{\,q+1\,}{2})=2\frac{\,\Gamma(\frac{\,p+1\,}{2})\,\Gamma(\frac{\,q+1\,}{2})\,}{\Gamma(\frac{\,p+q\,}{2}+1)},
\end{eqnarray*}
we obtain from (\ref{eqn:1.17}) that
\begin{eqnarray*}
\int_{S^{n-1}}\omega^\alpha\,d\omega&=&2\frac{\,\Gamma(m_n+\frac{1}{2})\,\Gamma(m_{n-1}+\frac{1}{2})\,}{\Gamma(m_n+m_{n-1}+1)}
\frac{\,\Gamma(m_n+m_{n-1}+1)\,\Gamma(m_{n-2}+\frac{1}{2})\,}{\Gamma(m_n+m_{n-1}+m_{n-2}+\frac{\,3\,}{2})}\cdots \\ 
&\phantom{a}&\phantom{\qquad\qquad\qquad\qquad\qquad\qquad\qquad\qquad}\times \frac{\,\Gamma(m_n+\dots+m_2+\frac{\,n-1\,}{2})\,\Gamma(m_1+\frac{1}{2})\,}{\Gamma(m_n+\dots+m_1+\frac{\,n\,}{2})}\\
&=&2\frac{\prod^n_{i=1}\,\Gamma(m_i+\frac{\,1\,}{2})\,}{\,\Gamma(|m|+\frac{\,n\,}{2})\,},\\
\end{eqnarray*}
which is the desired equality when all components of $\alpha$ are even.  In case when $\alpha$ contains at least one odd 
component, similar expression to (\ref{eqn:1.17}) readily yields $\displaystyle \int_{S^{n-1}}\omega^\alpha\,d\omega=0$,
since for $p\in\N$ and $q$ odd, $\displaystyle \int^\pi_0 \sin^p x\cos^q x\,dx=0.$
The proof of Lemma 3.2 is completed. 

\smallskip

We next give the proof of Lemma 3.3. The partial differentiations of the heat kernel $G_t(x)$ are written in terms of the Hermite 
polynomials $H_n(x)$ which are defined by 
$$\displaystyle 
H_n(x)=(-1)^n\,e^{x^2} \frac{d^n}{dx^n}\bigl( e^{-x^2}\bigr), \,\,\,n=0,1,2,\dots.
$$
In fact, as in \cite{Ch} we have, for a multi-index $\alpha=(\alpha_1,\dots,\alpha_n)\in(\N_0)^n$,
\begin{eqnarray*}\displaystyle 
\partial^\alpha_x G_t(x)&=&\frac{1}{(4\pi t)^{\frac{n}{2}}}\prod^n_{i=1}\, (-1)^{\alpha_i}\,e^{-\bigl(\frac{x_i}{2\sqrt{t}}\bigl)^2}
\Bigl(\frac{1}{\,2\sqrt{t}\,}\Bigr)^{\alpha_i}\,H_{\alpha_i}\bigl(\frac{x_i}{2\sqrt{t}}\bigr)\\
&=& G_t(x)\,\Bigl(-\frac{1}{\,2\sqrt{t}\,}\Bigr)^{|\alpha|}\prod^n_{i=1}H_{\alpha_i}\bigl(\frac{x_i}{2\sqrt{t}}\bigr)\\
&=& \tilde{G}_t (r)\,\Bigl(-\frac{1}{\,2\sqrt{t}\,}\Bigr)^{|\alpha|}\prod^n_{i=1}H_{\alpha_i}\bigl(\frac{r\,\omega_i}{2\sqrt{t}}\bigr).\\
\end{eqnarray*}
On the other hand, since $H_n(x)$ is given explicitly by 
$$\displaystyle 
H_n(x)=(2x)^n-\frac{\,n(n-1)\,}{1!}\,(2x)^{n-2}+\frac{\,n(n-1)(n-2)(n-3)\,}{2!}\,(2x)^{n-4}-\dots
$$
with the last term $\displaystyle (-1)^{\frac{n}{\,2\,}}\,\frac{\,n!\,}{\bigl(\frac{n}{\,2\,}\bigr)!}$ for even $n$ and 
 $\displaystyle (-1)^{\frac{n-1}{2}}\frac{n!}{\bigl(\frac{\,n-1\,}{2}\bigr)!} 2x$ for odd $n$ (see pp.91--93 in \cite{CoHi}), 
a straightforward calculation leads to the expression (\ref{eqn:3.2}).
 
\bigskip

$\bf{Appendix \, B.\,\, Proof \,of \,(\ref{eqn:1.d}).}$

\bigskip

It follows from (\ref{eqn:1.b}) and (\ref{eqn:1.a}) with $i=j$ that 
\begin{eqnarray*}
\M_\alpha(f)\,t^{*,\alpha}&=&-\frac{1}{\,(k+1)(k+2)\,}\M_{\alpha+2e_i}(f)-\frac{1}{\,2!\,}\M_{\alpha}(f)(x^{*,\alpha})^{2e_i}\\ 
&=&-\frac{1}{\,(k+1)(k+2)\,}\left(\M_{\alpha+2e_i}(f)-\frac{\,k+2\,}{2}\M_{\alpha+e_i}(f)\,x^{*,\alpha}\right).\\
\end{eqnarray*}
Here we used the relation such that
\begin{equation}\label{eqn:append1}
\M_{\alpha+e_i}(f)\,x^{*,\alpha}_i=(k+1)\M_\alpha(f)(x^{*,\alpha})^2,
\end{equation}
which is derived from (\ref{eqn:1.13}).
Hence, we have
\begin{equation}\label{eqn:append2}
n(k+1)(k+2)\M_\alpha(f)\,t^{*,\alpha}=-\sum^n_{i=1}\M_{\alpha+2e_i}(f)+\frac{\,k+2\,}{2}\sum^n_{i=1}\M_{\alpha+e_i}(f)\,
x^{*\alpha}_i.
\end{equation}
Then, referring to an elementary equality 
$$\displaystyle 
\frac{\,k+2\,}{2}=2s-\frac{s^2}{\,k+1\,},
$$
where 
$$\displaystyle
s=\frac{\,2(k+1)+\sqrt{2k(k+1)\,}\,}{2} 
$$
and using the relation (\ref{eqn:append1}), we find that (\ref{eqn:append2}) turns out to be 
\begin{eqnarray*}
n(k+1)(k+2)\M_\alpha(f)\,t^{*,\alpha}&=&\displaystyle 
-\sum^n_{i=1}\M_{\alpha+2e_i}(f)+2s\sum^n_{i=1}\M_{\alpha+e_i}(f)\,x_i^{*,\alpha}-
\frac{s^2}{\,k+1\,}\sum^n_{i=1}\M_{\alpha+e_i}(f)\,x^{*,\alpha}_i\\
&=& \displaystyle 
-\sum^n_{i=1}\M_{\alpha+2e_i}(f)+2s\sum^n_{i=1}\M_{\alpha+e_i}(f)\,x_i^{*,\alpha}
-s^2\sum^n_{i=1}\M_\alpha(f)(x_i^{*,\alpha})^2\\
%&=&\displaystyle -\sum^n_{i=1}\int_{\R^n}\left\{x_i^2-2s\,x_i\, x_i^{*,\alpha}+s^2(x_i^{*,\alpha})^2\right\}x^\alpha\,f(x)\,dx\\
&=&\displaystyle -\sum^n_{i=1}\int_{\R^n}(x_i-s\,x_i^{*,\alpha})^2\,x^\alpha\,f(x)\,dx.
\end{eqnarray*}
Now the proof of (\ref{eqn:1.d}) is completed.

\bigskip

$\bf{Appendix \, C.\,\, Proof \,of \, Proposition \,1.8.}$

Since $f\in C_c(\R^n)$, there exists a compact set $K$ in $\R^n$ such that $\mbox{supp}\, f\subset K$ and $f$ is regarded as a continuous function on $K$. 
Therefore, by Stone-Weierstrass theorem (see, for example,  pp. 138-141 in \cite{Fo}), we can see that there exists 
a sequence of polynomials $\{P_n(x)\}$ on $\R^n$ such that
$$
P_n(x)=\sum_{|\alpha|<\infty}\,c_{\alpha,\,n} \,x^\alpha, \,\,\,c_{\alpha,\,n}\in\R
$$
satisfying that
$$
\sup_{x\in K}\,|f(x)-P_n(x)|<\frac{\,1\,}{n}.
$$
Since $K$ is compact, $P_n(x) f(x)$ converges uniformly on $K$, so that
$$
\lim_{n\rightarrow\infty}\int_K\,P_n(x)f(x)\,dx=\int_K\,|f(x)|^2\,dx.
$$
On the other hand, under the assumption of Lemma 1.7, it is easy to see that 
$$
\int_K\,P_n(x) f(x)\,dx= \sum_{|\alpha|<\infty}\, c_{\alpha,\,n}\int_{\R^n}\, x^\alpha\,f(x)\,dx=0\,\,\mbox{for\,all} \, n\in\N.
$$
Consequently, we can conclude that 
$$
\int_{\R^n}\,|f(x)|^2\,dx=\int_K\,|f(x)|^2\,dx=0,
$$
so that $f(x)\equiv 0$. The proof of Lemma 1.7 is completed.

\bigskip

$\bf{Appendix \, D.\,\, An\, example \, of \, the \, initial \, data \, satisfying \, the \, condition \, (\ref{eqn:a})}.$

We give an example of nontrivial functions $f$ satisfying the condition (\ref{eqn:a}). Assume that  the functions $f$ take the form 
such that 
\begin{align}\label{eq:5.1}
f(x)=\prod_{i=1}^n \, a_i(x_i),
\end{align}
where
$$
a_i(x)=(c_{i,0}+c_{i,1}\,x) \, e^{-x^2}, \qquad c_{i,0}\in\R\setminus\{0\},\, c_{i,1} \in \R,  \;\;i=1,\dots , n.
$$
Put
\begin{align*}
\gamma_j=\int_{\R}x^{2j}\, e^{-x^2}\,dx,\,\,j=0,1,\dots,n.
\end{align*}
It is easily seen that
\begin{align}\label{eq:5.2}
\gamma_j=\frac{|2j-1|!!}{2^j}\sqrt{\pi},\;\;\; \;\; j=0,1,\dots,n.
\end{align}
%We note that $\gamma_0=\sqrt{\pi}$, $\gamma_1=\sqrt{\pi}/2$.
Then the  0-th and the first order moments of $f$ in (\ref{eq:5.1}) are calculated  as 
\begin{align*}
\int_{\R^n}\,f(x)\,dx&= \prod_{i=1}^n \int_{\R^n}(c_{i,0}+c_{i,1}x_h)\,e^{-x_i^2}\,dx_i=(\,\prod_{i=1}^n\,c_{i,0}\,)\,\gamma_0^n =
\pi^{\frac{n}{2}}\prod_{i=1}^n\,c_{i,0} \n
\end{align*}
and
\begin{align*}
\int_{\R^n}\,x_i\, f(x)\,dx&=\int_{\R^n}x_i\,(c_{i,0}+c_{i,1}x_i)\,e^{-x_i^2}\,dx_i \prod_{\substack{k=1\\ k \neq i}}^n \int_{\R^n}(c_{k,0}+c_{k,1}x_k)\,e^{-x_k^2}\,dx_k =c_{i,1}\, \gamma_1\,(\,\prod_{\substack{k=1 \\ k\neq i}}^n\,c_{k,0}\,)\,\gamma_0^{n-1}\\
& =\frac{\pi^{\frac{n}{2}}}{2}\,c_{i,1} \, (\prod_{\substack{k=1 \\ k\neq i}}^n c_{k,0}\,),\,\,i=1,\dots,n. 
\end{align*}
Therefore, the spatial shift $x_i^*$ defined by (\ref{eqn:1.6}) is given by % can be written as
\begin{equation}\label{eq:5.2a}
x_i^*=\frac{1}{\,2\,}\frac{c_{i,1}}{c_{i,0}}, \quad i=1,\dots,n.
\end{equation}
Next, we have a look at  (\ref{eqn:a}) with $i\neq j$. We find by (\ref{eq:5.2a}) that the two terms on the left-hand side of (\ref{eqn:a}) are given by
\begin{align}\label{eq:5.3}
\frac{1}{2}\int_{\R^n}x_ix_jf(x)\,dx
&=\frac{1}{2}\int_{\R^n}x_i\,(c_{i,0}+c_{i,1}x_i)\,e^{-x_i^2}\,dx_j \int_{\R^n}x_j\,(c_{j,0}+c_{j,1}x_j)\,e^{-x_j^2}\,dx_j  \\
&\hspace{170pt} \times \prod_{\substack{k=1 \\ k \neq i,j}}^n \int_{\R^n}(c_{k,0}+c_{k,1}x_k)\,e^{-x_k^2}\,dx_k  \n\\
&= \frac{1}{2}\,c_{i,1}\,c_{j,1}\, \gamma_1^2\,(\,\prod_{\substack{k=1 \\ k\neq i,j}}^n\,c_{k,0}\,)\,\gamma_0^{n-2}\n\\
&=\frac{\pi^{\frac{n}{2}}}{2^3}c_{i,1}\,c_{j,1}\,(\,\prod_{\substack{k=1 \\ k\neq i,j}}^n\,c_{k,0}\,), \n
\end{align}
and
\begin{align}\label{eq:5.4}
\frac{1}{2}\int_{\R^n}f(x)\,dx\,x_i^*\,x_j^*&=\prod_{k=1}^n\,c_{k,0} \, \,\pi^{\frac{n}{2}}\frac{1}{\,2\,}\frac{c_{i,1}}{c_{i,0}}\frac{1}{\,2\,}\frac{c_{j,1}}{c_{j,0}}   \\
&=\frac{\pi^{\frac{n}{2}}}{2^3}c_{i,1}\,c_{j,1}\,(\,\prod_{\substack{k=1 \\ k\neq i ,j}}^n\,c_{k,0}\,).\hspace{100pt} \n
\end{align}
Consequently, we find by (\ref{eq:5.3}) and (\ref{eq:5.4}) that the condition (\ref{eqn:a}) with $i\neq j$ is automatically satisfied for $f$ in (\ref{eq:5.1}). 
On the other hand, the left-hand side of (\ref{eqn:a}) with $i=j$ turns out 
\begin{align*}
&\frac{1}{2}\int_{\R^n}x_i^2\,f(x)\,dx-\frac{1}{2}\int_{\R^n}f(x)\,dx\,(x_i^*)^2  \\
&=\frac{1}{2}\int_{\R^n}x_i^2(c_{i,0}+c_{i,1}x_i)\,e^{-x_i^2}\,dx_i \prod_{\substack{k=1 \\ k \neq i}}^n \int_{\R^n}(c_{k,0}+c_{k,1}x_h)\,e^{-x_k^2}\,dx_k -\frac{1}{2}\prod_{k=1}^n\,c_{k,0} \, \pi^{\frac{n}{2}}\left(\frac{1}{\,2\,}\frac{c_{i,1}}{c_{i,0}}\right)^2 \n \\
&=\frac{1}{2}\,c_{i,0}\,\gamma_1 (\prod_{\substack{k=1 \\ k\neq i}}^n c_{k,0}\,)\,\gamma_0^{n-1}-\frac{1}{2}\prod_{k=1}^n\,c_{k,0} \, \pi^{\frac{n}{2}}\left(\frac{1}{\,2\,}\frac{c_{i,1}}{c_{i,0}}\right)^2 \n \\
&=\frac{\pi^{\frac{n}{2}}}{4}\prod_{k=1}^n\,c_{k,0}\left( 1- \frac{1}{\,2\,}\left(\frac{c_{i,1}}{c_{i,0}}\right)^2 \right). \n
\end{align*}
Accordingly,  if we take the coefficients $c_{i,0}$ and $c_{i,1}$ so that $\dis \frac{c_{i,1}}{\,c_{i,0}\,}=c$ for some constant $c\in\R$,  
for all $i=1,\dots,n$, we can see that the condition (\ref{eqn:a}) with $i=j$ holds for $\displaystyle c_0=\frac{\pi^{\frac{n}{2}}}{4}\prod_{k=1}^n\,c_{k,0}\left( 1- \frac{1}{\,2\,} c^2 \right)$. Then it is easy to see that the spatial shift $x^*$ defined by (\ref{eqn:1.6}) and  the time shift $t^*$ defined by (\ref{eqn:c}) is given as
\begin{align*}
x_i^*= \frac{c}{\,2\,}, \;\;\; i=1,\dots,n, \qquad t^*=-\frac{1}{\;4\;} \, \left(1-\frac{c^2}{\,2\,} \right)\,\,\,\mbox{for \, some \,constant}\,c\in\R.
\end{align*}
\\

\bigskip

$\bf{Appendix \, E.\,\, An\, example \, of \, the \, initial \, data \, satisfying \, the \, Condition \, A.}$

We give an example of nontrivial functions  $f$ satisfying Condition A  in the form such that 
\begin{align}\label{eq:5.5}
f(x)=\prod_{i=1}^n \, a_i(x_i),
\end{align}
where
$$
a_i(x)=\sum_{j=0}^{k+1}(c_{i,j} \, x^{j}) \, e^{-x^2}, \qquad c_{i,j} \in \R, \,i=1,\dots,n,\,j-0,1,\dots,k+1.
$$
Put
\begin{align*}
\X_i^{(m)}=\int_{\R}x_j^{m}\,a_i(x_j)\,dx_j.
\end{align*}
We note that, for $\al=(\al_1,\dots,\al_n)\in (\N_0)^n$ and $i,j=1,\dots,n$,  
\begin{align*}
\M_{\al}(f)=\prod_{h=1}^n \X_h^{(\al_h)}, \;\; 
\M_{\al+e_i}(f)=\prod_{h=1}^n \X_h^{(\al_h+\delta_{hi})} \;\; \mbox{and}\; \M_{\al+e_i+e_j}(f)=\prod_{h=1}^n \X_h^{(\al_h+\delta_{hi}+\delta_{hj})}.
\end{align*}
Thus, the spatial shifts $x^{*,\al}$, $\al\in\Lambda_k(f)$, defined by (\ref{eqn:1.13}) can be given by 
$$
x_i^{*,\al}=\frac{1}{\,k+1\,}\frac{\X_i^{(\al_i+1)}}{\X_i^{(\al_i)}}, \quad i=1,\dots,n.
$$
We next turn to look at (\ref{eqn:1.a}). The terms on the left-hand side of (\ref{eqn:1.a}) is represented by
\begin{align}\label{eq:5.6}
&\frac{1}{(k+1)(k+2)}\M_{\al+e_i+e_j}(f)-\frac{1}{2!}\M_{\al}(f)(x^{*,\al})^{e_i+e_j} \\
&\phantom{111111}=\frac{1}{(k+1)(k+2)}\prod_{h=1}^n\X_h^{(\al_h+\delta_{hi}+\delta_{hj})}-\frac{1}{2(k+1)^2}\prod_{h=1}^n\X_h^{(\al_h)}\frac{\X_i^{(\al_i+1)}}{\X_i^{(\al_i)}}\frac{\X_j^{(\al_j+1)}}{\X_j^{(\al_j)}}. \n 
%&\phantom{111111}=\left\{\frac{1}{(k+1)(k+2)}-\frac{1}{2(k+1)^2}\right\} \prod_{h=1}^n\X_h^{(\al_h+\delta_{hi}+\delta_{hj})}.\n
\end{align}
Therefore, Condition A can be rewritten as follows. 
For all $\alpha\in\Lambda_k(f)$, there exist constants $c_\alpha\in\R$ depending only on $\alpha$ 
such that  
\begin{multline}\label{eqn:re1.a}
\frac{1}{(k+1)(k+2)}\prod_{h=1}^n\X_h^{(\al_h+\delta_{hi}+\delta_{hj})}-\frac{1}{2(k+1)^2}\prod_{h=1}^n\X_h^{(\al_h)}\frac{\X_i^{(\al_i+1)}}{\X_i^{(\al_i)}}\frac{\X_j^{(\al_j+1)}}{\X_j^{(\al_j)}}=c_\al \delta_{ij}\\
 \mbox{for all}\,\,i,j=1,\dots,n.
%\frac{1}{2(k+1)^2(k+2)}\left\{ 2(k+1) \X_i^{(\al_i+2)}-(k+2) \frac{\left( \X_i^{(\al_i+1)} \right)^2}{\X_i^{(\al_i)}}\right\}\prod_{\substack{h=1 \\ h \neq i}}^n\X_h^{(\al_h)}=c_\al \,\,\, \mbox{for all}\,\,i=1,\dots,n.
\end{multline}
While, for all $\alpha\notin\Lambda_k(f)$ with $|\alpha|=k$, the moments of the $f$ satisfy that 
\begin{equation}\label{eqn:re1.c}
\prod_{h=1}^n\X_h^{(\al_h+\delta_{hi})}=\prod_{h=1}^n\X_h^{(\al_h+\delta_{hi}+\delta_{hj})}=0\,\,\,\mbox{for all} \,\,i,j=1,\dots,n.
%\mbox{for \, all}\,\,i=1,\dots,n,\,\,\mbox{and\, for\, all }\,\alpha'\in (\N_0)^n\,\mbox{with}\,|\alpha'|=2.
\end{equation}
When Condition A holds, the time shifts $t^{*,\al}$, $\alpha\in\Lambda_k(f)$, given by (\ref{eqn:1.b}) can be seen to obey 
\begin{align*}
t^{*,\al}=-\biggl\{\frac{1}{(k+1)(k+2)}\frac{\X_i^{(\al_i+2)}}{\X_i^{(\al_i)}}-\frac{1}{2(k+1)^2}\biggl(\frac{\X_i^{(\al_i+1)}}{\X_i^{(\al_i)}}\biggr)^2\biggr\}.
\end{align*}
because
\begin{multline*}
\frac{1}{(k+1)(k+2)}\prod_{h=1}^n\X_h^{(\al_h+2\delta_{hi})}-\frac{1}{2(k+1)^2}\prod_{h=1}^n\X_h^{(\al_h)}\biggl(\frac{\X_i^{(\al_i+1)}}{\X_i^{(\al_i)}}\biggr)^2\\
=
\M_\al(f)\biggl\{\frac{1}{(k+1)(k+2)}\frac{\X_i^{(\al_i+2)}}{\X_i^{(\al_i)}}-\frac{1}{2(k+1)^2}\biggl(\frac{\X_i^{(\al_i+1)}}{\X_i^{(\al_i)}}\biggr)^2\biggr\}.
\end{multline*}
Therefore, when $k=1$, taking in particular 
$\dis (c_{1,0}\,,c_{1,1}\,,c_{1,2})=(\mathcal{C}_1 , \pm\sqrt{2}\,\mathcal{C}_1,-2\,\mathcal{C}_1)$ 
and $\dis(c_{i,0}\,,c_{i,1}\,,c_{i,2})=(\mathcal{C}_i,\,0,\,-\frac{2}{3}\,\mathcal{C}_i)$, $i=2,\dots,n$, with arbitrary constants 
$\mathcal{C}_i\in \R\setminus\{0\}$, $i=1,2,\dots,n$, it is easy to see that the function $f$ in (\ref{eq:2.15}) 
satisfies Condition A and the spatial and time shifts $x^{*,\alpha}$ and $t^{*,\alpha}$, $\alpha\in\Lambda_1(f)$,  are given by
$$ \Lambda_{1}(f)=\{e_1\}, \,\, x_1^{*,e_1}=\mp\sqrt{2},\,\,x_i^{*,e_1}=0\,\,
\mbox{for}\,\,i=2,\dots,n,\,\,t^{*,e_1}=0.
$$

\bigskip\noindent
{\bf Acknowledgments.} 
The authors would like to express their thanks to Professor Hiroshi Ishii for providing information on the paper \cite{Ch}. 
The work of the first author is supported by JST, the establishment of
university fellowships towards the creation of science technology innovation,
Grant Number JPMJFS 2127.
The work of the second author is partially supported by 
JSPS KAKENHI Grant Number 22K03375 . The authors declare no conflicts of interest.

\end{document}